\documentclass[oneside,english]{amsart}
\usepackage[T1]{fontenc}
\usepackage[latin9]{inputenc}
\usepackage{geometry}
\usepackage{color}
\usepackage{babel}
\usepackage{verbatim}
\usepackage{mathrsfs}
\usepackage{amstext}
\usepackage{amsthm}
\usepackage{amssymb}
\usepackage[all]{xy}
\usepackage[unicode=true,pdfusetitle,
 bookmarks=true,bookmarksnumbered=false,bookmarksopen=false,
 breaklinks=false,pdfborder={0 0 1},backref=false,colorlinks=true]
 {hyperref}

\makeatletter

\newcommand{\mathcircumflex}[0]{\mbox{\^{}}}

\numberwithin{equation}{section}
\numberwithin{figure}{section}
\theoremstyle{plain}
\newtheorem*{thm*}{\protect\theoremname}
\theoremstyle{plain}
\newtheorem{thm}{\protect\theoremname}[section]
\theoremstyle{definition}
\newtheorem{defn}[thm]{\protect\definitionname}
\theoremstyle{plain}
\newtheorem{conjecture}[thm]{\protect\conjecturename}
\theoremstyle{remark}
\newtheorem{rem}[thm]{\protect\remarkname}
\theoremstyle{plain}
\newtheorem{lem}[thm]{\protect\lemmaname}
\theoremstyle{plain}
\newtheorem{prop}[thm]{\protect\propositionname}
\theoremstyle{plain}
\newtheorem{cor}[thm]{\protect\corollaryname}
\theoremstyle{remark}
\newtheorem{notation}[thm]{\protect\notationname}
\theoremstyle{definition}
\newtheorem{example}[thm]{\protect\examplename}
\theoremstyle{definition}
\newtheorem*{condition*}{\protect\conditionname}
\theoremstyle{plain}
\newtheorem{question}[thm]{\protect\questionname}

\usepackage[all]{xy}
\usepackage{url}

\makeatother

\providecommand{\conditionname}{Condition}
\providecommand{\conjecturename}{Conjecture}
\providecommand{\corollaryname}{Corollary}
\providecommand{\definitionname}{Definition}
\providecommand{\examplename}{Example}
\providecommand{\lemmaname}{Lemma}
\providecommand{\notationname}{Notation}
\providecommand{\propositionname}{Proposition}
\providecommand{\questionname}{Question}
\providecommand{\remarkname}{Remark}
\providecommand{\theoremname}{Theorem}

\begin{document}
\title{Rank 2 Local Systems, Barsotti-Tate Groups, and Shimura Curves}
\author{Raju Krishnamoorthy}
\begin{abstract}
We construct a descent-of-scalars criterion for $K$-linear abelian
categories. Using advances in the Langlands correspondence due to
Abe, we build a correspondence between certain rank 2 local systems
and certain Barsotti-Tate groups on complete curves over a finite
field. We conjecture that such Barsotti-Tate groups ``come from''
a family of fake elliptic curves. As an application of these ideas,
we provide a criterion for being a Shimura curve over $\mathbb{F}_{q}$.
Along the way we formulate a conjecture on the field-of-coefficients
of certain compatible systems. 

\tableofcontents{}
\end{abstract}

\maketitle

\section{Introduction}

Let $C/\mathbb{F}_{q}$ be a smooth affine curve with compactification
$\overline{C}$ and let $\pi\colon E_{C}\rightarrow C$ be a non-isotrivial
family of elliptic curves. Then $R^{1}\pi_{*}\mathbb{Q}_{l}$ is a
rank 2 $l$-adic local system on $C$ with infinite monodromy around
some $\infty\in\overline{C}\backslash C$, cyclotomic determinant,
and all Frobenius traces in $\mathbb{Q}$. A startling consequence
of Drinfeld's first work on the Langlands correspondence is a converse: 
\begin{thm*}
\label{Theorem:elliptic}(Drinfeld) Let $C/\mathbb{F}_{q}$ be a smooth
affine curve with compactification $\overline{C}$. Let $\mathscr{L}$
be a rank 2 irreducible $\overline{\mathbb{Q}}_{l}$-local system
on $C$ such that 
\begin{itemize}
\item $\mathscr{L}$ has infinite monodromy around some $\infty\in\overline{C}\backslash C$,
\item $\mathscr{L}$ has determinant $\overline{\mathbb{Q}}_{l}(-1)$, and
\item the field of Frobenius traces of $\mathscr{L}$ is $\mathbb{Q}$.
\end{itemize}
Then $\mathscr{L}$ comes from a family of elliptic curves: there
exists a map $f$
\[
\xymatrix{ & \mathcal{E}\ar[d]^{\pi}\\
C\ar[r]^{f} & \mathcal{M}_{1,1}
}
\]
 such that $\mathscr{L}\cong f^{*}(R^{1}\pi_{*}\overline{\mathbb{Q}}_{l})$.
Here $\mathcal{M}_{1,1}$ is the moduli of elliptic curves with universal
elliptic curve $\mathcal{E}$.
\end{thm*}
See \cite[Proposition 19, Remark 20]{snowden2018constructing} for
how to recover this result from Drinfeld's work.
\begin{defn}
\label{Definition:moduli_fake_elliptic_curves}Let $D$ be an indefinite
non-split quaternion algebra over $\mathbb{Q}$ of discriminant $d$
and let $\mathcal{O}_{D}$ be a fixed maximal order. Let $k$ be a
field with $\textrm{char}(k)\nmid d$. A \emph{quaternionic abelian}
\emph{surface }is a pair $(A,i)$ of an abelian surface $A/k$ together
with an injective ring homomorphism $i:\mathcal{O}_{D}\rightarrow\text{End}_{k}(A)$.
\end{defn}

$D$ has a canonical involution, which we denote $\iota$. Pick $t\in\mathcal{O}_{D}$
with $t^{2}=-d$. Then there is another associated involution $*$
on $D$:
\[
x^{*}:=t^{-1}x^{\iota}t
\]
There is a unique principal polarization $\lambda$ on $A$ such that
the Rosati involution restricts to $*$ on $\mathcal{O}_{D}$. We
refer to the triple $(A,\lambda,i)$ as a \emph{fake elliptic curve},
suppressing the dependence on $t\in\mathcal{O}_{D}$. 

Just as one can construct a modular curve parameterizing elliptic
curves, there is a Shimura curve $X^{D}$ parameterizing fake elliptic
curves with multiplication by $\mathcal{O}_{D}$. Over the complex
numbers, these are compact hyperbolic curves. Explicitly, if one chooses
an isomorphism $D\otimes\mathbb{R}\cong M_{2\times2}(\mathbb{R})$,
consider the image of $\Gamma=\mathcal{O}_{D}^{1}$ of elements of
$\mathcal{O}_{D}^{*}$ of norm 1 (for the standard norm on $\mathcal{O}_{D}$)
inside of $SL(2,\mathbb{R})$. \textbf{$\Gamma$ }acts properly discontinuously
and cocompactly on $\mathbb{H}$. The quotient $X^{D}=[\mathbb{H}/\Gamma]$
is the complex Shimura curve associated to $\mathcal{O}_{D}$. In
fact, $X^{D}$ has a canonical integral model and may therefore be
reduced modulo $p$ for almost all $p$ \cite{buzzard1997integral}.
In analogy of Theorem \ref{Theorem:elliptic}, we pose the following
conjecture.
\begin{conjecture}
\label{Conjecture:rank_2_fake_elliptic_curve}Let $C/\mathbb{F}_{q}$
be a smooth projective curve. Let $\mathscr{L}$ be a rank 2 irreducible
$\overline{\mathbb{Q}}_{l}$-local system on $C$ such that 
\begin{itemize}
\item $\mathscr{L}$ has infinite geometric monodromy,
\item $\mathscr{L}$ has determinant $\overline{\mathbb{Q}}_{l}(-1)$, and
\item the field of Frobenius traces of $\mathscr{L}$ is $\mathbb{Q}$.
\end{itemize}
Then, $\mathscr{L}$ comes from a family of fake elliptic curves:
there exists an indefinite quaternion algebra $D/\mathbb{Q}$, a moduli
space of fake elliptic curves $X^{D}$with universal family $\mathcal{A}\xrightarrow{\pi}X^{D}$,
and a map $f$:
\[
\xymatrix{ & \mathcal{A}\ar[d]^{\pi}\\
C\ar[r]^{f} & X^{D}
}
\]
 such that $\mathscr{L}^{\oplus2}\cong f^{*}(R^{1}\pi_{*}\overline{\mathbb{Q}}_{l})$.
\end{conjecture}

In this article we prove the following, which perhaps provides some
evidence for \ref{Conjecture:rank_2_fake_elliptic_curve}. 
\begin{thm*}
(Theorem \ref{Theorem:p-div_companion_versal-1}) Let $C/\mathbb{F}_{q}$
be a smooth, geometrically connected, proper curve with $q$ a square.
There is a natural bijection between the following two sets
\[
\left\{ \begin{alignedat}{1} & \overline{\mathbb{Q}}_{l}\text{-local systems }\mathscr{L}\text{ on }C\text{ such that}\\
 & \bullet\mathscr{L}\text{ is irreducible of rank 2}\\
 & \bullet\mathscr{L}\text{ has trivial determinant}\\
 & \bullet\text{The Frobenius traces are in }\mathbb{Q}\\
 & \bullet\mathscr{L}\text{ has infinite image,}\\
 & \text{up to isomorphism}
\end{alignedat}
\right\} \longleftrightarrow\left\{ \begin{aligned} & p\text{-divisible groups }\mbox{\ensuremath{\mathscr{G}}}\text{ on }C\text{ }\text{ such that }\\
 & \bullet\mathscr{G}\text{ has height 2 and dimension 1}\\
 & \bullet\mathscr{G}\text{ is generically versally deformed}\\
 & \bullet\mathbb{D}(\mathscr{G})\text{ has all Frobenius traces in }\mathbb{Q}\\
 & \bullet\mathscr{G}\text{ has ordinary and supersingular points,}\\
 & \text{up to isomorphism}
\end{aligned}
\right\} 
\]
such that if $\mathscr{L}$ corresponds to $\mathscr{G}$, then $\mathscr{L}(-1/2)$
is compatible with the $F$-isocrystal $\mathbb{D}(\mathscr{G})\otimes\mathbb{Q}$.
\end{thm*}
Theorem \ref{Theorem:p-div_companion_versal-1} arose from the following
question: is there a ``purely group-theoretic'' characterization
of proper Shimura curves over $\mathbb{F}_{q}$? As an application
of our techniques, we have the following criterion, motivated largely
by \cite{margulis1991,mochizuki1996ordinary,mochizuki1998correspondences,xia2013deformation}.
\begin{thm*}
(Theorem \ref{Theorem:2_pullbacks_L_isomorphic_versally}) Let $X\overset{f}{\leftarrow}Z\overset{g}{\rightarrow}X$
be an étale correspondence of smooth, geometrically connected, proper
curves without a core over $\mathbb{F}_{q}$ with $q$ a square. Let
$\mathscr{L}$ be a $\overline{\mathbb{Q}}_{l}$-local system on $X$
as in Theorem \ref{Theorem:p-div_companion_versal-1} with $f^{*}\mathscr{L}\cong g^{*}\mathscr{L}$
as local systems on $Z$. Suppose the $\mathscr{G}\rightarrow X$
constructed in Theorem \ref{Theorem:p-div_companion_versal-1} is
everywhere versally deformed. Then $X$ is the reduction modulo $p$
of a Shimura curve.
\end{thm*}
Note that the criterion only involves varieties in characteristic
$p$ and makes no mention of a ``family of abelian varieties''. The
lifting process involved in Theorem \ref{Theorem:2_pullbacks_L_isomorphic_versally}
is due to Xia \cite{xia2013deformation}, and it gives one of the
\emph{canonical lifts} of Mochizuki \cite{mochizuki1996ordinary}.
In \cite[Problem 10]{oortopen}, Mochizuki asks when these canonical
lifts are defined over $\overline{\mathbb{Q}}$. The techniques of
Section \ref{Section:Shimura_Curves} give a criterion: see Corollary
\ref{Corollary:p-div_shimura} and Remark \ref{Remark:mochizuki}.

In joint work-in-progress with Mao Sheng, we have a strategy to prove
that in the context of Theorem \ref{Theorem:2_pullbacks_L_isomorphic_versally},
$\mathscr{L}$ indeed comes from a family of fake elliptic curves,
verifying a very special case of Conjecture \ref{Conjecture:rank_2_fake_elliptic_curve}.
More precisely, the method should imply that $X$ is the reduction
modulo $p$ of a moduli space of fake elliptic curves and $\mathscr{L}$
is the induced universal system. See Remark \ref{Remark:strategy_fake_elliptic_curves}
for a detailed outline of the strategy, which uses $p$-adic nonabelian
Hodge theory.

Finally, as a basic observation from our analysis, we prove the following:
\begin{thm*}
(Theorem \ref{Theorem:F-Isoc-Q-Finite-Monodromy}) Let $X$ be a smooth,
geometrically connected curve over $\mathbb{F}_{q}$. Let $\mathcal{E}\in\textbf{F-Isoc}^{\dagger}(X)$
be an overconvergent $F$-isocrystal on $X$ with coefficients in
$\mathbb{Q}_{p}$ that is rank 2, absolutely irreducible, and has
trivial determinant. Suppose further that the field generated by Frobenius
traces on $\mathcal{E}$ is $\mathbb{Q}$. Then $\mathcal{E}$ has
finite monodromy.
\end{thm*}
We make some remarks about Theorem \ref{Theorem:p-div_companion_versal-1}.
\begin{rem}
The condition that $q$ is a square is to ensure that the character
$\overline{\mathbb{Q}}_{l}(1/2)$ on $\text{Gal}(\mathbb{F}/\mathbb{F}_{q})$
has Frobenius acting as an integer. Therefore $\mathscr{L}(-1/2)$
also has Frobenius traces in $\mathbb{Q}$ and determinant $\overline{\mathbb{Q}}_{l}(-1)$,
exactly as in Conjecture \ref{Conjecture:rank_2_fake_elliptic_curve}.
\end{rem}

\begin{rem}
Drinfeld's work on unramified $GL_{2}$ Langlands implies that if
$\mathscr{L}$ is as in Conjecture \ref{Conjecture:rank_2_fake_elliptic_curve},
then there exists a family of surfaces over a Zariski open $U\subset C\times C$
\[
g:\mathcal{S}\rightarrow U
\]
such that $(\mathscr{L}\boxtimes\mathscr{L}^{*})|_{U}$ is a direct
summand of $R^{2}g_{*}\overline{\mathbb{Q}}_{l}$. Our conjecture
is that $\mathscr{L}$ may be realized inside a ``motive of weight
1'' over $C$. (See Question \ref{Question:algebraize_G} for a related
question.) 
\end{rem}

\begin{rem}
The principle input in Theorem \ref{Theorem:p-div_companion_versal-1}
is Abe's construction of $p$-adic companions on curves over a finite
field. Given this, there are two new ingredients.
\begin{enumerate}
\item The $p$-adic companion $\mathcal{E}\in\textbf{F-Isoc}(X)_{\overline{\mathbb{Q}}_{p}}$
of $\mathscr{L}$ may be descended to $\mathbb{Q}_{p}$; in other
words, a certain Brauer obstruction ``automatically'' vanishes.
This is contained in Lemma \ref{Lemma:Main_Isocrystal_Descent} and
Proposition \ref{Proposition:Rank_2_twisted_descent}.
\item There is a \emph{canonical lattice} inside of $\mathcal{E}$; this
is characterized by the associated Barsotti-Tate group being \emph{generically
versally deformed}. This is contained in Lemma \ref{Lemma:unique_generically_deformed_BT_group}.
\end{enumerate}
\end{rem}

\begin{rem}
In particular, given $\mathscr{L}$ as in Theorem \ref{Theorem:p-div_companion_versal-1},
there is a natural effective divisor (possibly empty) $R$ on $C$,
supported at those $c\in|C|$ over which $\mathscr{G}$ is not versally
deformed. If Conjecture \ref{Conjecture:rank_2_fake_elliptic_curve}
is true, $R$ is the ramification divisor of a generically separable
map
\[
C\rightarrow X^{D}
\]
realizing $\mathscr{L}$ as coming from a family of fake elliptic
curves. We do not know how to construct this effective divisor directly
either from $\mathscr{L}$ or its $p$-adic companion $\mathcal{E}$
without constructing $\mathscr{G}$.
\end{rem}

We now briefly discuss the sections. The goal of Sections \ref{Section:Galois_descent_for_abelian_categories}-\ref{Section:Kodaira_spencer}
is to prove Theorem \ref{Theorem:p-div_companion_versal-1}.
\begin{itemize}
\item Section \ref{Section:Galois_descent_for_abelian_categories} sketches
extension-of-scalars and Galois descent for $K$-linear abelian categories.
\item Section \ref{Section:2_cocycle_obstruction_for_descent} constructs
a Brauer-class obstruction for descending certain objects in an abelian
$K$-linear category.
\item Section \ref{Section:F_crystals} briefly reviews background on $F$-crystals
and $F$-isocrystals. Many of the statements are surely well-known,
but we couldn't always find a reference.
\item Section \ref{Section:Compatible_systems} is a brief discussion of
coefficient objects and several results of Abe and Lafforgue.
\item Section \ref{Section:F_isocrystals_local_systems_on_curves} contains
the application of the descent criterion above in Proposition \ref{Proposition:Rank_2_twisted_descent}.
\item Section \ref{Section:Kodaira_spencer} discusses the deformation theory
of BT groups. We prove that there is a unique such $\mathscr{G}$
that is \emph{generically versally deformed, }thereby proving Theorem
\ref{Theorem:p-div_companion_versal-1}.
\end{itemize}
BT groups are generally non-algebraic, but here we can show there
are only finitely many that occur in Theorem \ref{Theorem:p-div_companion_versal-1}
via the Langlands correspondence. In Section \ref{Section:Algebraization_and_finite_monodromy},
we speculate on when height 2, dimension 1 BT groups $\mathscr{G}$
come from a family of abelian varieties. 

Finally, in Section \ref{Section:Shimura_Curves}, we discuss applications
to ``characterizing'' certain Shimura curves over $\mathbb{F}_{q}$.

\textbf{Acknowledgements.} This work is based on Chapter 3 of my PhD
thesis at Columbia University. I am very grateful to Johan de Jong,
my former thesis advisor, for guiding this project and for countless
inspiring discussions. Ching-Li Chai read my thesis very carefully
and provided many valuable corrections and remarks; I thank him. I
also thank Marco d'Addezio, Hélène Esnault, and Ambrus Pál for discussions
on the subject of this article. I thank the anonymous referee, who
provided helpful comments on the article. This work was partially
funded by an NSF postdoctoral fellowship, Grant No. DMS-1605825.

\section{Conventions, Notation, and Terminology}

For convenience, we explicitly state conventions and notations. These
are in full force unless otherwise stated.
\begin{enumerate}
\item A \emph{curve} $C/k$ is a separated geometrically integral scheme
of dimension 1 over $k$.
\item A \emph{variety} $X/k$ is a geometrically integral $k$-scheme of
finite type.
\item A morphism of curves $X\rightarrow Y$ over $k$ is a morphism of
$k$-schemes that is non-constant, finite, and generically separable.
\item $p$ is a prime number and $q=p^{d}$.
\item $\mathbb{F}$ is a fixed algebraic closure of $\mathbb{F}_{p}$.
\item If $k$ is perfect, $W(k)$ denotes the Witt vectors and $\sigma$
the canonical lift of Frobenius. $K(k)\cong W(k)\otimes\mathbb{Q}$. 
\item All $p$-adic valuations are normalized such that $v_{p}(p)=1$. 
\item If $\mathcal{C}$ is a $K$-linear category and $L/K$ is an algebraic
extension, we denote by $\mathcal{C}_{L}$ the base-changed category.
\item Let $X/k$ be a smooth variety over a perfect field. Then the category
$\textbf{F-Isoc}^{\dagger}(X)$ is the $\mathbb{Q}_{p}$-linear Tannakian
category of \emph{overconvergent $F$-isocrystals on X.}
\item If $\mathscr{G}\rightarrow X$ is a $p$-divisible group, we denote
by $\mathbb{D}(\mathscr{G})$ the \emph{contravariant} Dieudonné crystal
attached to $\mathscr{G}$. We use the phrases $p$-divisible group
and Barsotti-Tate group interchangably. 
\end{enumerate}

\section{\label{Section:Galois_descent_for_abelian_categories}Extension of
Scalars and Galois descent for abelian categories}

We discuss extension-of-scalars of $K$-linear categories; in the
case of $K$-linear abelian categories, we state Galois descent. We
also carefully discuss absolute irreducibility and semi-simplicity. 
\begin{defn}
Let $\mathcal{C}$ be a $K$-linear additive category, where $K$
is a field. Let $L/K$ be a finite field extension. We define the
\emph{base-changed category} $\mathcal{C}_{L}$ as follows.
\end{defn}

\begin{itemize}
\item Objects of $\mathcal{C}_{L}$ are pairs $(M,f)$, where $M$ is an
object of $\mathcal{C}$ and $f:L\rightarrow\text{End}_{\mathcal{C}}M$
is a homomorphism of $K$-algebras. We call such an $f$ \emph{an
$L$-structure on $M$}.
\item Morphisms of $\mathcal{C}_{L}$ are morphisms of $\mathcal{C}$ that
are compatible with the $L$-structure.
\end{itemize}
If $\mathcal{C}$ is an abelian category, so is $\mathcal{C}_{L}$
\cite{sosna2014scalar}. Moreover, if $\mathcal{C}$ is a $K$-linear
Tannakian category (not necessarily assumed to be neutral), then $\mathcal{C}_{L}$
is an $L$-linear Tannakian category \cite[Th\'eor\`eme 5.4]{deligne2014semi}.

\begin{rem}
Deligne defines $\mathcal{C}_{L}$ in a slightly different way \cite[5]{deligne2014semi}.
Given an object $X$ of $\mathcal{C}$ and a finite dimensional $K$-vector
space $V$, define $V\otimes X$ to represent the functor $Y\mapsto V\otimes_{K}\text{Hom}_{K}(Y,X)$.
Given a finite extension $L/K$, an object with $L$-structure is
defined to be an object $X$ of $\mathcal{C}$ together with a morphism
$L\otimes X\rightarrow X$ in $\mathcal{C}$ satisfying certain natural
conditions. Then the base-changed category $\mathcal{C}_{L}$ is defined
to be the category of objects equipped with an $L$-structure and
with morphisms respecting the $L$-structure.
\end{rem}

Note that there is an induction functor: $\text{Ind}_{K}^{L}:\mathcal{C}\rightarrow\mathcal{C}_{L}$.
Let $\{\alpha\}$ be a basis for $L/K$. 
\[
\text{Ind}_{K}^{L}M=(\bigoplus_{\alpha}M,f)
\]
where $f$ is induced by the action of $L$ on the $\{\alpha\}$.
In Deligne's formulation, $\text{Ind}_{K}^{L}M$ is just $L\otimes M$
and the $L$-structure is the multiplication map $L\otimes_{K}L\otimes M\rightarrow L\otimes M$.
The induction functor has a natural right adjoint: $\text{Res}_{K}^{L}$
which simply forgets about the $L$-structure. If $M$ is an object
of $\mathcal{C}$, we sometimes denote by $M_{L}$ the object $\text{Ind}_{K}^{L}M$
for shorthand. If $C$ is a $K$-linear abelian category, induction
and restriction are exact functors.

\begin{rem}
\label{Remark:Absolute_Induction}The induction functor allows us
to describe more general extensions-of-scalars. Namely, let $E$ be
an \emph{algebraic} field extension of $K$. We define $\mathcal{C}_{E}$
to be the 2-colimit of the categories $\mathcal{C}_{L}$ where $L$
ranges over the subfields of $E$ finite over $K$. In particular,
we can define a category $\mathcal{C}_{\overline{K}}$ and an induction
functor 
\[
Ind_{K}^{\overline{K}}:\mathcal{C}\rightarrow\mathcal{C}_{\overline{K}}.
\]
\end{rem}

\begin{defn}
Let $\mathcal{C}$ be a $K$-linear abelian category. We say that
an object $M$ is \emph{absolutely irreducible }if $Ind_{K}^{\overline{K}}(M)$
is an irreducible object of $\mathcal{C}_{\overline{K}}$.
\end{defn}

If $L/K$ is a Galois extension, note that we have a natural (strict)
action of $G:=\text{Gal}(L/K)$ on the category $\mathcal{C}_{L}$
by twisting the $L$-structure. That is, if $g\in G$, $^{g}(C,f):=(C,f\circ g^{-1})$.
The group $G$ ``does nothing'' to maps: a map $\phi:(C,f)\rightarrow(C',f')$
in $\mathcal{C}_{L}$ is just a map $\phi_{K}:C\rightarrow C'$ in
$\mathcal{C}$ that commutes with the $L$-actions, and $g\in G$
acts by fixing the underlying $\phi_{K}$ while twisting the underlying
$L$-structures $f$, $f'$: $^{g}\phi:(C,f\circ g^{-1})\rightarrow(C',f'\circ g^{-1})$.

If $\lambda\in L$ is considered as a scalar endomorphism of $M$,
then the endomorphism $^{g}\lambda:\ ^{g}M\rightarrow\ ^{g}M$ is
the scalar $g(\lambda)$. If $\mathcal{C}$ is a $K$-linear rigid
abelian $\otimes$ category, then this action is compatible with the
inherited rigid abelian $\otimes$ structure on $\mathcal{C}_{L}$.
Similarly, if $F:\mathcal{C}\rightarrow\mathcal{D}$ is a $K$-linear
functor between $K$-linear categories, we have a canonical extension
functor $F_{L}:\mathcal{C}_{L}\rightarrow\mathcal{D}_{L}$ and $^{g}F_{L}(M)\cong F_{L}(^{g}M)$.
\begin{defn}
Suppose $\mathcal{C}$ is a abelian $K$-linear category and $L/K$
is a finite Galois extension with group $G$. Let $\mathcal{C}_{L}$
be the base-changed category. We define the \emph{category of descent
data}, $(\mathcal{C}_{L})^{G}$, as follows. The objects of $(\mathcal{C}_{L})^{G}$
are pairs $(M,\{c_{g}\})$ where $M$ is an object of $\mathcal{C}_{L}$
and the $c_{g}:M\rightarrow\ ^{g}M$ are a collection of isomorphisms
for each $g\in G$ that satisfies the cocycle condition. The morphisms
of $(\mathcal{C}_{L})^{G}$ are maps $(f_{g}:\ ^{g}M\rightarrow\ ^{g}N)$
that intertwine the $c_{g}$. 
\end{defn}

\begin{lem}
\label{Lemma:Galois_Descent_Abelian_Categories}If $\mathcal{C}$
is an abelian $K$-linear category and $L/K$ is a finite Galois extension,
then Galois descent holds. That is, $\mathcal{C}$ is equivalent to
the category $(\mathcal{C}_{L})^{G}$.
\end{lem}

\begin{proof}
This is \cite[Lemma 2.7]{sosna2014scalar}.
\end{proof}
We say objects and morphisms in the essential image of $\text{Ind}_{K}^{L}$
\emph{descend}.

\section{\label{Section:2_cocycle_obstruction_for_descent}A 2-cocycle obstruction
for descent}

Suppose $\mathcal{C}$ is an abelian $K$-linear category and $L/K$
a finite Galois extension with group $G$. Let $\mathcal{C}_{L}$
be the base-changed category. Suppose an object $M\in\text{Ob}(\mathcal{C}_{L})$
is isomorphic to all of its twists by $g\in G$ and that the natural
map $L\tilde{\rightarrow}\text{End}_{\mathcal{C}_{L}}M$ is an isomorphism.
(This latter restriction will be relaxed later in the important Remark
\ref{Remark:Central_Cocycle}.) We will define a cohomology class
$\xi_{M}\in H^{2}(G,L^{*})$ such that $\xi_{M}=0$ if and only if
$M$ descends to $K$. This construction is well-known in representation
theory.
\begin{defn}
\label{Definition:cocycle}For each $g\in G$ pick an isomorphism
$c_{g}:M\rightarrow\ ^{g}M$. The function $\xi_{M,c}:G\times G\rightarrow L^{*}$,
depending on the choices $\{c_{g}\}$, is defined as follows: 
\[
\xi_{M,c}(g,h)=c_{gh}^{-1}\circ\ ^{g}c_{h}\circ c_{g}\in Aut_{\mathcal{C}_{L}}(M)\cong L^{*}
\]
\end{defn}

\begin{prop}
The function $\xi_{M,c}$ is a 2-cocycle.
\end{prop}

\begin{proof}
We need to check 
\[
^{g_{1}}\xi(g_{2},g_{3})\xi(g_{1},g_{2}g_{3})=\xi(g_{1}g_{2},g_{3})\xi(g_{1},g_{2})
\]
We may think of the right hand side as a scalar function $M\rightarrow M$,
which allows us to write it as
\begin{align*}
\xi(g_{1}g_{2},g_{3})\xi(g_{1},g_{2})= & c_{g_{1}g_{2}g_{3}}^{-1}\circ\ ^{g_{1}g_{2}}c_{g_{3}}\circ c_{g_{1}g_{2}}\circ c_{g_{1}g_{2}}^{-1}\circ\ ^{g_{1}}c_{g_{2}}\circ c_{g_{1}}\\
= & c_{g_{1}g_{2}g_{3}}^{-1}\circ\ ^{g_{1}}(^{g_{2}}c_{g_{3}}\circ c_{g_{2}})\circ c_{g_{1}}\\
= & c_{g_{1}g_{2}g_{3}}^{-1}\circ\ ^{g_{1}}c_{g_{2}g_{3}}\circ\ ^{g_{1}}(c_{g_{2}g_{3}}^{-1}\circ\ ^{g_{2}}c_{g_{3}}\circ c_{g_{2}})\circ c_{g_{1}}\\
= & c_{g_{1}g_{2}g_{3}}^{-1}\circ\ ^{g_{1}}c_{g_{2}g_{3}}\circ\ ^{g_{1}}\xi(g_{2},g_{3})\circ c_{g_{1}}\\
= & c_{g_{1}g_{2}g_{3}}^{-1}\circ\ ^{g_{1}}c_{g_{2}g_{3}}\circ c_{g_{1}}\circ\ ^{g_{1}}\xi(g_{2},g_{3})\\
= & \xi(g_{1},g_{2}g_{3})\ ^{g_{1}}\xi(g_{2},g_{3})
\end{align*}
In the penultimate line, we may commute the $c_{g_{1}}$ and $^{g_{1}}\xi(g_{2,}g_{3})$
because the latter is in $L$.
\end{proof}
\begin{rem}
If $\xi_{M,c}=1$, then the collection $\{c_{g}\}$ form a descent
datum for $M$ which is effective by Galois descent for abelian categories,
Lemma \ref{Lemma:Galois_Descent_Abelian_Categories}.
\end{rem}

\begin{prop}
\label{Proposition:coboundary_descends}Let $\mathcal{C}$ be an abelian
$K$-linear category and $L/K$ a finite Galois extension with group
$G$. Let $M\in\text{Ob}(\mathcal{C}_{L})$ such that 
\begin{itemize}
\item $M\cong\ ^{g}M$ for all $g\in G$
\item The natural map $L\hookrightarrow End_{L}(M)$ is an isomorphism.
\end{itemize}
If $\xi_{M,c}$ is a coboundary, then $M$ is in the essential image
of $\text{Ind}_{K}^{L}$, i.e. $M$ descends.
\end{prop}

\begin{proof}
If $\xi_{M,c}$ is a coboundary, there exists a function $\alpha:G\rightarrow L^{*}$
such that
\[
\xi_{M,c}(g,h)=\frac{^{g}\alpha(h)\alpha(g)}{\alpha(gh)}
\]
 Now, set $c'_{g}=\frac{c_{g}}{\alpha(g)}:M\rightarrow\ ^{g}M$ and
note that the $c'_{g}$ are a descent datum for $M$ because the associated
$\mbox{\ensuremath{\xi}}_{M,c'}=1$.
\end{proof}
\begin{prop}
\label{Proposition:unique_cocycle}Given $M\in\mathcal{C}_{L}$ as
in Proposition \ref{Proposition:coboundary_descends} and two choices
$\{c_{g}\}$ and $\{c'_{g}\}$ of isomorphisms, $\xi_{M,c}$ and $\xi_{M,c'}$
differ by a coboundary and thus give the same class in $H^{2}(G,L^{*})$.
We may therefore unambiguously write $\xi_{M}$ for the cohomology
class associated to $M$.
\end{prop}

\begin{proof}
Note that $(c'_{g})^{-1}\circ c_{g}:M\rightarrow M$ is in $L^{*}$.
This ratio will be a function $\alpha:G\rightarrow L^{*}$ exhibiting
the ratio $\frac{\xi_{M,c}}{\xi_{M,c'}}$ as a coboundary.
\end{proof}
\begin{cor}
\label{Corollary:brauer_class_descends}Let $\mathcal{C}$ be an abelian
$K$-linear category, let $L/K$ be a finite Galois extension with
group $G$, and let $\mathcal{C}_{L}$ be the base-changed category.
Let $M$ be an object of $\mathcal{C}_{L}$ such that

\begin{enumerate}
\item $M\cong\ ^{g}M$ for all $g\in G$ and
\item the natural map $L\rightarrow\text{End}_{\mathcal{C}_{L}}(M)$ is
an isomorphism.
\end{enumerate}
Then the cocycle $\xi_{M}$ (as in Definition \ref{Definition:cocycle})
is 0 in $H^{2}(K,\mathbb{G}_{m})$ if and only if $M$ descends.

\end{cor}

\begin{proof}
Combine Propositions \ref{Proposition:coboundary_descends} and \ref{Proposition:unique_cocycle}.
\end{proof}
\begin{rem}
\label{Remark:Central_Cocycle}We did not have to assume that $L\rightarrow End_{\mathcal{C}_{L}}(M)$
was an isomorphism for a cocycle to exist. A necessary assumption
is that there exists a collection $\{c_{g}\}$ of isomorphisms such
that $\xi_{M,c}(g,h)\in L^{*}$ for all $g,h\in G$. The key is that
$H^{2}$ exists as long as the coefficients are abelian. Note, however,
that in this level of generality there is no guarantee that cohomology
class is unique: it depends very much on the choice of the isomorphisms
$\{c_{g}\}$. Therefore, this technique \emph{will not be adequate
}to prove that objects do not descend; we can only prove than an object
does descend by finding a collection $\{c_{g}\}$ whose associated
$\xi_{c}$ is a coboundary.
\end{rem}

\begin{rem}
Note that if $\mathcal{C}_{\mathbb{R}}$ is the category of real representations
of a compact group and $\mathcal{C}_{\mathbb{C}}$ is the complexification,
namely the category of complex representations of a compact group,
then this 2-cocycle has a more classical name: ``Frobenius-Schur
Indicator''. It tests whether an irreducible complex representation
of a compact group with real character can be defined over $\mathbb{R}$.
If not, the representation is called quaternionic.
\end{rem}

Now, suppose $\mathcal{C}$ is a $K$-linear rigid abelian $\otimes$
category. If $M\in\text{Ob}(\mathcal{C}_{L})$ such that $^{g}M\cong M$
for all $g\in G$ and $L\rightarrow End_{\mathcal{C}_{L}}M$ is an
isomorphism, then the same is true for $M^{*}$. Moreover, choosing
$\{c_{g}\}$ for $M$ gives the natural choice of $\{(c_{g}^{*})^{-1}\}$
for $M^{*}$ so $\xi_{M}^{-1}=\xi_{M^{*}}$ .

In general, if $M$ and $N$ are as above with choices of isomorphisms
$\{c_{g}:M\rightarrow\ ^{g}M\}$ and $\{d_{g}:N\rightarrow\ ^{g}N\}$
with associated cohomology classes $\xi_{M}$ and $\xi_{N}$ respectively,
then we can cook up a cohomology class to possibly detect whether
$M\otimes N$ descends, $\xi_{M}\xi_{N}$, using the isomorphism $c_{g}\otimes d_{g}:M\otimes N\rightarrow\ ^{g}M\otimes\ ^{g}N\cong\ ^{g}(M\otimes N)$.
This is interesting because in general $M\otimes N$ might have endomorphism
algebra larger than $L$; in particular, we weren't guaranteed the
existence of a cohomology class $\xi_{M\otimes N}$, as discussed
in Remark \ref{Remark:Central_Cocycle}.
\begin{lem}
Let $\mathcal{C}$ be a $K$-linear rigid abelian $\otimes$ category,
$L/K$ a finite Galois extension with group $G,$ and $\mathcal{C}_{L}$
the base-changed category. Let $M\in Ob(\mathcal{C}_{L})$ have endomorphism
algebra $L$ and suppose $^{g}M\cong M$ for all $g\in G$. Then $\underbar{End}(M)\cong M\otimes M^{*}$
descends to $\mathcal{C}_{K}$.
\end{lem}

\begin{proof}
As noted above, if $\xi$ is the cocycle associated to $M$, then
$\xi^{-1}$ is the cocycle associated to $M^{*}$. Then $1=\xi\xi^{-1}$
is a cocycle associated to $M\otimes M^{*}$, whence it descends.
\end{proof}
\begin{rem}
A related classical fact: let $V$ be a finite dimensional complex
representation $V$ of a compact group $G$. Then the representation
$\underbar{End}(V)\cong V\otimes V^{*}$ is defined over $\mathbb{R}$.
\end{rem}

We now give two criteria for descent. Though the second is strictly
more general than the first, the hypotheses are more complicated and
we found it helpful to separate the two.
\begin{lem}
\label{Lemma:Easy_Descent}Let $F:\mathcal{C}\rightarrow\mathcal{D}$
be a $K$-linear functor between abelian $K$-linear categories. Let
$L/K$ be a finite Galois extension with group $G$ and let $F_{L}:\mathcal{C}_{L}\rightarrow\mathcal{D}_{L}$
be the base-changed functor. Let $M\in\text{Ob}(\mathcal{C}_{L})$
be an object such that $^{g}M$ is isomorphic to $M$ for $g\in G$
and $End_{\mathcal{C}_{L}}M\cong L$. Suppose $End_{\mathcal{D}_{L}}F_{L}(M)\cong L$.
Then $M$ descends to $\mathcal{C}$ if and only if $F_{L}(M)$ descends
to $\mathcal{D}$.
\end{lem}

\begin{proof}
The natural map $Hom_{\mathcal{C}_{L}}(M,\ ^{g}M)\rightarrow Hom_{\mathcal{D}_{L}}(F_{L}(M),\ ^{g}F_{L}(M))$
is an isomorphism. Hence choosing isomorphisms $\{d_{g}:F_{L}(M)\rightarrow\ ^{g}F_{L}(M)\cong F_{L}(^{g}M)\}$
is the same as choosing isomorphisms $\{c_{g}:M\rightarrow\ ^{g}M\}$.
Therefore the cocycles $\xi_{F(M)}$ and $\xi_{M}$ are the same.%
\end{proof}
\begin{lem}
\label{Lemma:Filtered_Descent}Let $F:\mathcal{C}\rightarrow\mathcal{D}$
be a $K$-linear functor between $K$-linear abelian categories and
let $L/K$ be a finite Galois extension with group $G$. Let $F_{L}:\mathcal{C}_{L}\rightarrow\mathcal{D}_{L}$
be the base-changed functor. Suppose $M\in\text{Ob}(\mathcal{C}_{L})$
with $L\cong End_{\mathcal{C}_{L}}M$ and $^{g}M\cong M$ for all
$g\in G$. Further suppose $F_{L}(M)\cong N_{1}\oplus N_{2}$ satisfying
the following two conditions:

\begin{itemize}
\item $L\cong End_{\mathcal{D}_{L}}N_{1}$ and
\item there is no nonzero morphism $N_{1}\rightarrow\ ^{g}N_{2}$ in $\mathcal{D}_{L}$
for any $g\in G$.
\end{itemize}
Then $M$ (and $N_{2}$) descend if and only if $N_{1}$ descends.

\end{lem}

\begin{proof}
The composition

\begin{align*}
\text{Hom}_{\mathcal{C}_{L}}(M,\ ^{g}M)\rightarrow & \text{Hom}_{\mathcal{D}_{L}}(F_{L}(M),\ ^{g}F_{L}(M))\\
\cong & \text{Hom}_{\mathcal{D}_{L}}(N_{1}\oplus N_{2},\ ^{g}N_{1}\oplus\ ^{g}N_{2})\rightarrow\text{Hom}_{\mathcal{D}_{L}}(N_{1},\ ^{g}N_{1})
\end{align*}
is a homomorphism of $L$-vector spaces. In fact, the map is nonzero
because an isomorphism $c_{g}:M\rightarrow\ ^{g}M$ is sent to the
isomorphism $F_{L}(c_{g})$. By the second assumption, this projects
to an isomorphism $N_{1}\rightarrow\ ^{g}N_{1}$. By the first assumption
on $N_{1}$, the total composition is an isomorphism. Therefore a
collection $\{n_{g}:N_{1}\rightarrow\ ^{g}N_{1}\}$ is canonically
the same as a collection $\{m_{g}:M\rightarrow\ ^{g}M\}$ and thus
we have $\xi_{M}=\xi_{N_{1}}$ .
\end{proof}
We now examine the relation between the rank of $M$ and the order
of its induced Brauer class $[\xi_{M}]\in H{{}^2}(K,\mathbb{G}_{m})$
when $\mathcal{C}$ is assumed to be a $K$-linear Tannakian category.
Recall that Tannakian categories are not necessarily neutral, i.e.,
they do not always admit a fiber functor to $\text{Vect}_{K}$. For
the remainder of this section, $K$ is supposed to be a field of characteristic
0. Recall that Tannakian categories have a natural notion of rank
\cite{deligne2007categories}.) If $P$ is an object of rank 1, there
is a natural diagram
\[
\text{End}(P)\cong P\otimes P^{*}\rightleftarrows\mbox{K}
\]
where the top arrow is evaluation (i.e. the trace) and the bottom
arrow comes from the $K$-vector-space structure of $\text{End}(P)$.
As $P\otimes P^{*}$ has rank 1, these two maps identify $\text{End}(P)$
isomorphically with $K$. 

\begin{prop}
\label{Proposition:Neutral_Rank_1_Descent}Let $K$ be a field of
characteristic 0 and let $\mathcal{C}$ be a neutral $K$-linear Tannakian
category. Let $L/K$ be a finite Galois extension and let $\mathcal{C}_{L}$
be the base-changed category. Let $P\in Ob(\mathcal{C}_{L}$) be a
rank-1 object. If $^{g}P\cong P$ for all $g\in G$ then $P$ descends.
\end{prop}

\begin{proof}
Let $F:\mathcal{C}\rightarrow\text{Vect}_{K}$ be a fiber functor
and denote by $F_{L}$ the base-changed fiber functor. By definition,
$P$ being rank 1 means that $F_{L}(P)$ is a rank 1 $L$-vector space,
so $L\cong\text{End}\big(F_{L}(P)\big)\cong\text{End}(P)$. All vector
spaces descend, so by Lemma \ref{Lemma:Easy_Descent}, $P$ descends
as well.
\end{proof}
\begin{lem}
Let $K$ be a field of characteristic 0, let $\mathcal{C}$ be a $K$-linear
Tannakian category, and let $L/K$ be a finite Galois extension. Suppose
$M\in\text{Ob}(\mathcal{C}_{L})$ has rank $r$, $^{g}M\cong M$ for
all $g\in G$, and $L\cong End_{\mathcal{C}_{L}}M$. If $\overset{r}{\bigwedge}M$
descends (which is automatically satisfied if $\mathcal{C}$ is neutral
by Proposition \ref{Proposition:Neutral_Rank_1_Descent}), then $\xi_{M}$
is $r$-torsion in $H^{2}(G,L^{*}).$ In particular, there exists
a degree $r$ extension of $K$ over which $M$ is defined.
\end{lem}

\begin{proof}
Pick $\{c_{g}:M\rightarrow\ ^{g}M\}$ giving the cohomology class
$\xi_{M}$. The isomorphisms 
\[
\{c_{g}^{\otimes r}:M^{\otimes r}\rightarrow(^{g}M)^{\otimes r}\cong\ ^{g}(M^{\otimes r})\}
\]
preserve the space of anti-symmetric tensors and restrict to give
isomorphisms
\[
c_{g}^{\otimes r}:\bigwedge^{r}M\rightarrow\ ^{g}\bigwedge^{r}M
\]
The cohomology class associated to $c_{g}^{\otimes r}$ is $\xi^{r}$,
and this cohomology class is unique because $L\cong End\bigwedge M$
as $\bigwedge M$ is a rank 1 object. As we assumed $\bigwedge M$
descends, we deduce that $\xi^{r}=0\in H^{2}(G,L^{*})$.
\end{proof}

\section{\label{Section:F_crystals}$F$-crystals}

In this section, we set out our conventions about $F$-crystals and
$F$-isocrystals. We further recall several results which are surely
well-known but which we could not find documented in the literature.
We recommend that the reader skip this section and refer to it when
necessary. For a meta-reference, see \cite{kedlaya2016notes}.

Let $X/k$ be a smooth scheme of finite type over a perfect field
$k$. Berthelot has defined the absolute crystalline site on $X$:
for a ``modern'' reference, see \cite[TAG 07I5]{stacks-project}.
(We implicitly take the crystalline site with respect to $W(k)$ without
further comment; in other words, in the formulation of the Stacks
Project, $S=\text{Spec}(W(k))$ with the canonical PD structure.)
Let $Crys(X)$ be the category of crystals in \emph{finite coherent
$\mathcal{O}_{X/W(k)}$-modules}. By functoriality of the crystalline
topos, the absolute Frobenius $Frob_{X}:X\rightarrow X$ gives a functor
$Frob_{X}^{*}:Crys(X)\rightarrow Crys(X)$.
\begin{notation}
Let $X/k$ be a smooth scheme over a perfect field.
\begin{itemize}
\item An $F$-crystal in finite, locally free modules on $X$ is a pair
$(M,F)$ where $M$ is a crystal in finite, locally free modules and
$F\colon Frob_{X}^{*}M\rightarrow M$ is an isogeny. The $\mathbb{Z}_{p}$-linear
category of $F$-crystals in finite, locally free modules is denoted
as $F\text{\text{-Crys}}(X)$.
\item A Dieudonné crystal on $X$ is a triple $(M,F,V)$ where $(M,F)$
is an $F$-crystal in finite, locally free modules and $V\colon M\rightarrow Frob_{X}^{*}M$
is an isogeny such that $F\circ V=p$ and $V\circ F=p$.
\item The category $\textbf{F-Isoc}(X)$ denotes the $\mathbb{Q}_{p}$-linear
Tannakian category of $F$-isocrystals. This is the same as the isogeny
category of $F$-crystals in finite, coherent modules.
\item The category $\textbf{F -Isoc}^{\dagger}(X)$ denotes the $\mathbb{Q}_{p}$-linear
Tannakian category of overconvergent $F$-isocrystals.
\end{itemize}
\end{notation}

We record the following elementary lemma, which provides an explicit
description of $\textbf{F-Isoc}(k)_{L}$ with $L/\mathbb{Q}_{p}$
an algebraic extension.
\begin{prop}
\label{Proposition:Coefficients_in_L}Let $k$ be a perfect field
and let $L/\mathbb{Q}_{p}$be an algebraic extension. Let $K(k):=\text{Frac}(W(k))$
denote the Fraction field of ring of Witt vectors and let $\sigma$
denote the canonical lift of Frobenius to $K(k)$. Then the category
$\textbf{F-Isoc}(k)_{L}$ is equivalent to the category of pairs $(V,F)$
where $V$ is a finite free module over $K(k)\otimes_{\mathbb{Q}_{p}}L$
and $F:V\rightarrow V$ is a $\sigma\otimes1$-linear bijective map.
The rank of $(V,F)$ is the rank of $V$ as a free $K(k)\otimes L$-module.
\end{prop}

\begin{rem}
Note that $K(k)\otimes L$ is \emph{not necessarily a field}; rather,
it is a direct product of fields and $\sigma\otimes Id$ permutes
the factors. It is a field if and only if $L$ and $K(k)$ are linearly
disjoint over $\mathbb{Q}_{p}$. This occurs, for instance, if $L$
is totally ramified over $\mathbb{Q}_{p}$ or if the maximal finite
subfield of $k$ is $\mathbb{F}_{p}$.
\end{rem}

\begin{proof}
First of all, we immediately reduce to the case when $L/\mathbb{Q}_{p}$
is a finite extension. By definition, an object $((V',F'),f)$ of
$\textbf{F-Isoc}(k)_{L}$ consists of $(V',F')$ an $F$-isocrystal
on $k$ and a $\mathbb{Q}_{p}$-algebra homomorphism $f:L\rightarrow End_{\textbf{F-Isoc}(k)}(V',F')$.
Recall that $V'$ is a finite dimensional $K(k)$ vector space and
$F'$ is a $\sigma$-linear bijective map. This gives $V'$ the structure
of a finite module (not a priori free) over $K(k)\otimes_{\mathbb{Q}_{p}}L$.
The bijection $F'$ commutes with the action of $L$, hence $F'$
is $\sigma\otimes1$-linear. We need only prove that $V'$ is a free
$K(k)\otimes_{\mathbb{Q}_{p}}L$-module.

Let $L^{\circ}$ be the maximal unramified subfield of $L$ and let
$M$ be the maximal subfield of $L^{\circ}$ contained in $K(k)$.
This notion is unambiguous: $M$ is the unramified extension of $\mathbb{Q}_{p}$
with residue field the intersection of the maximal finite subfield
of $k$ and the residue field of the local field $L$. Note that $L$
and $K(k)$ are linearly disjoint over $M$, so $K(k)\otimes_{M}L$
is a field. Let $r$ be the degree of the extension $M/\mathbb{Q}_{p}$.
Then $K(k)\otimes_{\mathbb{Q}_{p}}L$ is the direct product $\prod_{i=1}^{r}(K(k)\otimes_{M}L)_{i}$
and $\sigma\otimes1$ permutes the factors transitively. Because of
this direct product decomposition, $V'$ can be written as $\prod_{i=1}^{r}V'_{i}$
with each $V'_{i}$ a $K(k)\otimes_{M}L$-vector space. As $F'$ is
$\sigma\otimes1$ linear and bijective, $F'$ transitively permutes
the factors $V'_{i}$ and hence the dimension of each $V'_{i}$ as
a $K(k)\otimes_{M}L$ vector space is the same. This implies that
$V$ is a free $K(k)\otimes_{\mathbb{Q}_{p}}L$-module.
\end{proof}
\begin{defn}
Fix once and for all a compatible family $(p^{\frac{1}{n}})\in\overline{\mathbb{Q}}_{p}$
of roots of $p$. Let $\overline{\mathbb{Q}}_{p}(-\frac{a}{b})$ be
the following rank 1 object in $\textbf{F-Isoc}(\mathbb{F}_{p})_{\overline{\mathbb{Q}}_{p}}$
(using the description furnished by Proposition \ref{Proposition:Coefficients_in_L})
\[
(\overline{\mathbb{Q}}_{p}<v>,*p^{\frac{a}{b}})
\]
where $v$ is a basis vector. Abusing notation, given any perfect
field $k$ of characteristic $p$, we similarly denote the pullback
to $k$ by $\overline{\mathbb{Q}}_{p}(-\frac{a}{b})$. In terms of
Proposition \ref{Proposition:Coefficients_in_L}, it is given by the
rank 1 module $K(k)\otimes\overline{\mathbb{Q}}_{p}<v>$ with $F(v)=p^{\frac{a}{b}}v$
and extended $\sigma\otimes1$-linearly.
\end{defn}

\begin{cor}
Let $M=(V,F)$ be a rank 1 object of $\textbf{F-Isoc}(k)_{L}$ with
unique slope $\lambda=\frac{s}{r}$. Then $r|\deg[L:\mathbb{Q}_{p}]$.
\end{cor}

\begin{proof}
Consider the object $M'=\text{Res}_{\mathbb{Q}_{p}}^{L}M$. This is
an isoclinic $F$-isocrystal on $k$ of rank $\deg[L:\mathbb{Q}_{p}]$
with slope $\lambda$. We may suppose $k$ is algebraically closed;
then by the Dieudonné-Manin decomposition, $M'$ is isomorphic to
the direct sum of several copies of the unique simple $F$-isocrystal
$E^{\lambda}$ on $k$ with slope $\lambda$. The rank of $E^{\lambda}$
is $r$, so the rank of $M'$ is divisible by $r$.
\end{proof}
We now specialize our discussion to $F$-isocrystals over finite fields.
For the remainder of this section, $\mathbb{F}_{q}$ is a finite field
of $q$ elements, $\mathbb{Z}_{q}:=W(\mathbb{F}_{q})$, and $\mathbb{Q}_{q}:=\mathbb{Z}_{q}[1/p]$.
As usual, $\sigma$ denotes a lift of the absolute Frobenius. When
we use the phrase \emph{$p$-adic valuation}, it is always normalized
so that $v_{p}(p)=1$.
\begin{prop}
Let $(V,F)$ be an $F$-isocrystal on $\mathbb{F}_{q}$ with coefficients
in $L$ a $p$-adic local field. Then $F^{d}$ is a $\mathbb{Q}_{q}\otimes L$-linear
endomorphism of $V$. Let $P_{F}(t)\in\mathbb{Q}_{q}\otimes L[t]$
be given as 
\[
P_{F}(t):=\det(1-F^{d}t)|_{V}
\]
Then $P_{F}(t)\in L[t]$.
\end{prop}

\begin{proof}
Let $\tilde{P}$ be the characteristic polynomial of $F^{d}$. It
is equivalent to prove $\tilde{P}\in L[t]$. The ring automorphism
$\sigma\otimes1$ has order $d$ so $F^{d}$ is a linear endomorphism
on the free $\mathbb{Q}_{q}\otimes L$-module $V$. The polynomial
$\tilde{P}$ a priori has coefficients in $\mathbb{Q}_{q}\otimes L$,
so we must show that the coefficients of $\tilde{P}(t)$ are invariant
under $\sigma\otimes1$. Recall that the coefficients of the characteristic
polynomial of an operator $A$ are, up to sign, the traces of the
exterior powers of $A$. As there is a notion of $\bigwedge$ for
$F$-isocrystals it is enough to show that $\text{trace}(F^{d})$
is invariant under $\sigma\otimes1$.

To do this, pick a $\mathbb{Q}_{q}\otimes L$ basis $\{v_{i}\}$ of
$V$. Let $S:=(s_{ij})$ be the ``matrix'' of $F$ in this basis,
i.e. $F(v_{i})=\sum_{j}s_{ij}v_{j}$. Then an easy computation shows
that the matrix of $F^{d}$ in this basis is given by 
\[
(^{\sigma^{d-1}\otimes1}S)(^{\sigma^{d-2}\otimes1}S)\cdots(^{\sigma\otimes1}S)(S)
\]

Then $^{\sigma\otimes1}\text{trace}(F^{d})=\text{trace}(^{\sigma\otimes1}(F^{d}))=\text{trace\ensuremath{\big(}}(S)(^{\sigma^{d-1}\otimes1}S)\cdots(^{\sigma^{2}\otimes1}S)(^{\sigma\otimes1}S)\big)=\text{trace}(F^{d})$
because $\text{trace}(AB)=\text{trace}(BA)$.
\end{proof}
\begin{prop}
\label{Proposition:char_poly_of_isoc}Let $(V,F)$ be an $F$-isocrystal
on $\mathbb{F}_{q}$ with coefficients in $L$. Then $F^{d}$ is a
$\mathbb{Q}_{q}\otimes L$-linear endomorphism of $V$ with characteristic
polynomial $P_{F}(t)\in L[t]$. The slopes of $(V,F)$ are $\frac{-1}{d}$
times the $p$-adic valuations of the roots of $P_{F}(t)$.
\end{prop}

\begin{proof}
This follows immediately from the fact that the diagram
\[
\bigoplus_{\lambda}\textbf{F-Isoc}(k)^{\lambda}\cong\textbf{F-Isoc}(k)\rightleftarrows\textbf{F-Isoc}(k)_{L}\cong\bigoplus_{\lambda}\textbf{F-Isoc}(k)_{L}^{\lambda}
\]
respects the slope decomposition together with the remark after \cite[Lemma 1.3.4]{katz1979slope}.
\end{proof}

We now classify rank 1 objects $(V,F)\in\text{Ob }\textbf{F-Isoc}(\mathbb{F}_{q})_{L}$,
again using the explicit description in Proposition \ref{Proposition:Coefficients_in_L}.
As discussed above, the eigenvalue of $F^{d}$ is in $L$. The slogan
of Proposition \ref{Proposition:Rank_1_Eigenvalue} is that this eigenvalue
determines $(V,F)$ up to isomorphism. This allows us to classify
semi-simple $F$-isocrystals on $\mathbb{F}_{q}$.

Let $v$ be a free generator of $V$ over $\mathbb{Q}_{q}\otimes_{\mathbb{Q}_{p}}L$.
Then $F(v)=av$ with $a\in(\mathbb{Q}_{q}\otimes_{\mathbb{Q}_{p}}L)^{*}$
and $F^{d}(v)=(\text{Nm}_{\mathbb{Q}_{q}\otimes L/L}a)v$ where the
norm is taken with respect to the cyclic Galois morphism $L\rightarrow\mathbb{Q}_{q}\otimes L$.
Therefore, any $\lambda\in L$ that is in the image of the norm map
$\text{Nm}:(\mathbb{Q}_{q}\otimes L)^{*}\rightarrow L^{*}$ can be
realized as the eigenvalue of $F^{d}$. To prove that $F^{d}$ uniquely
determines a rank 1 $F$-isocrystal, we need only prove that there
is a unique rank 1 $F$-isocrystal with $F^{d}$ the identity map.
\begin{prop}
\label{Proposition:Rank_1_Eigenvalue}Let $(V,F)$ be a rank 1 $F$-isocrystal
on $\mathbb{F}_{q}$ with coefficients in $L$. Suppose $F^{d}$ is
the identity map. Then $(V,F)$ is isomorphic to the trivial $F$-isocrystal,
i.e. there is a basis vector $v\in V$ such that $F(v)=v$.
\end{prop}

\begin{proof}
Suppose $F(v)=\lambda v$ where $\lambda\in(\mathbb{Q}_{q}\otimes L)^{*}$.
Then $F^{d}(v)=\text{Nm}_{\mathbb{Q}_{q}\otimes L/L}(\lambda)v$,
where the Norm map is defined with respect to the Galois morphism
of algebras $L\hookrightarrow Q\otimes L$. This Galois group is cyclic,
generated by $\sigma\otimes1$. By assumption, we have $\text{Nm}_{\mathbb{Q}_{q}\otimes L/L}(\lambda)=1$.
Now, $F(av)=(^{\sigma\otimes1}a)\lambda v$ and we want to find an
$a\in(\mathbb{Q}_{q}\otimes L)^{*}$ such that $F(av)=av$, i.e. $\frac{a}{^{\sigma\otimes1}a}=\lambda$
where $\lambda$ has norm 1. This is guaranteed by Hilbert's Theorem
90 for the cyclic morphism of algebras $L\hookrightarrow Q\otimes L$.
\end{proof}
\begin{cor}
\label{Corollary:Rank_1_Descent}Let $M=(V,F)$ be a rank-1 $F$-isocrystal
on $\mathbb{F}_{q}$ with coefficients in $L$. If the eigenvalue
of $F^{d}$ lives in finite index subfield $K\subset L$ and is a
norm in $K$ with regards to the algebra homomorphism $K\rightarrow\mathbb{Q}_{q}\otimes K$,
then $M$ descends to $K$.
\end{cor}

\begin{proof}
Let the eigenvalue of $F^{d}$ be $\lambda$ and let $a\in\mathbb{Q}_{q}\otimes K$
have norm $\lambda.$ Consider the rank 1 object $E$ of $\textbf{F-Isoc}(\mathbb{F}_{q})_{K}$
given on a basis element $e$ by $F(e)=a^{-1}e$. Then $E_{L}\otimes M$
has $F^{d}$ being the identity map. Proposition \ref{Proposition:Rank_1_Eigenvalue}
implies that $E_{L}\otimes M$ is the trivial $F$-isocrystal, i.e.
that $M\cong(E_{L})^{*}\cong(E^{*})_{L}$. Therefore $M$ descends
as desired.
\end{proof}
\begin{example}
Consider the object $\overline{\mathbb{Q}}_{p}(-\frac{1}{2})$ of
$\textbf{F-Isoc}(\mathbb{F}_{p^{2}})_{\overline{\mathbb{Q}}_{p}}$.
The eigenvalue of $F^{2}$ is $p\in\mathbb{Q}_{p}$, but $p$ is not
in the image of the norm map and hence the object does not descend
to an object of $\textbf{F-Isoc}(\mathbb{F}_{p^{2}})$. One can also
see this by virtue of the fact that no non-integral fraction can occur
as the slope of a rank 1 object of $\textbf{F-Isoc}(k)$. Note that
the object does in fact descend to $\textbf{F-Isoc}(\mathbb{F}_{p})_{\mathbb{Q}_{p}(\sqrt{p})}$.
\end{example}

\begin{rem}
\label{Remark:Twists_of_fractional_tate_motives}Consider $\overline{\mathbb{Q}}_{p}(-\frac{a}{b})$
as an object of $\textbf{F-Isoc}(\mathbb{F}_{q})_{\overline{\mathbb{Q}}_{p}}$
where $\frac{a}{b}$ is in lowest terms. It is isomorphic to its Galois
twists by the group $\text{Gal}(\overline{\mathbb{Q}}_{p}/\mathbb{Q}_{p})$
if and only if $b|d$ where $q=p^{d}$. Indeed, the isomorphism class
of this $F$-isocrystal is determined by the eigenvalue of $F^{d}$,
which is $p^{\frac{ad}{b}}$. This is in $\mathbb{Q}_{p}$ if and
only if $b\mid d$.
\end{rem}

Corollary \ref{Corollary:Rank_1_Descent} poses a natural question.
Let $K$ be a $p$-adic local field and let $q=p^{d}$. Which elements
of $K$ are in the image of the norm map?
\[
\text{Norm}:(\mathbb{Q}_{q}\otimes K)^{*}\rightarrow K^{*}
\]
If $K=\mathbb{Q}_{p}$, then $K^{*}\cong p^{\mathbb{Z}}\times\mathbb{Z}_{p}^{*}$,
and the image of the norm map is exactly $p^{d\mathbb{Z}}\times\mathbb{Z}_{p}^{*}$.
More generally, if $K/\mathbb{Q}_{p}$ is a totally ramified extension
with uniformizer $\varpi$, then $K^{*}\cong\varpi^{\mathbb{Z}}\times\mathcal{O}_{K}^{*}$,
$\mathbb{Q}_{q}\otimes K$ is a field that is unramified over $K$,
and the image of the norm map is exactly $\varpi^{d\mathbb{Z}}\times\mathcal{O}_{K}^{*}$.
On the other extreme, if $K\cong\mathbb{Q}_{q}$, then the following
proposition shows that the norm map is surjective.
\begin{prop}
\label{Proposition:cyclic_norm_surjective}Let $L/K$ be a cyclic
extension of fields, with $\text{Gal}(L/K)$ generated by $g$ and
of order $n$. Consider the induced Galois morphism of algebras: $L\rightarrow L\otimes_{K}L$.
The image of the norm map for this extension is surjective.
\end{prop}

\begin{proof}
$L\otimes_{K}L\cong\prod_{x\in G}L$, where the $L$-algebra structure
is given by the first (identity) factor. Then $g$ acts by cyclically
shifting the factors and the norm of an element $(\dots,l_{x},\dots)$
is just $\prod_{x\in G}l_{x}$. Therefore the norm map is surjective.
\end{proof}
\begin{example}
We have the following strange consequence. Consider the object $\overline{\mathbb{Q}}_{p}(-\frac{1}{2})$
of $\textbf{F-Isoc}(\mathbb{F}_{p^{2}})_{\overline{\mathbb{Q}}_{p}}$.
It descends to  $\textbf{F-Isoc}(\mathbb{F}_{p^{2}})_{\mathbb{Q}_{p^{2}}}$:
$F^{2}$ has unique eigenvalue $p$, and $p$ is a norm for the algebra
homomorphism $\mathbb{Q}_{p^{2}}\hookrightarrow\mathbb{Q}_{p^{2}}\otimes\mathbb{Q}_{p^{2}}$
by Proposition \ref{Proposition:cyclic_norm_surjective}, so we may
apply Corollary \ref{Corollary:Rank_1_Descent}.
\end{example}

In general the image of the norm map is rather complicated to classify.
However, in our case we can say the following.
\begin{lem}
\label{Lemma:units_are_norms_unramified}Let $K$ be a $p$-adic local
field and let $q=p^{d}$. Then $\mathcal{O}_{K}^{*}$ is in the image
of the norm map $(\mathbb{Q}_{q}\otimes K)^{*}\rightarrow K^{*}$.
\end{lem}

\begin{proof}
$\mathbb{Q}_{q}\otimes K\cong\prod K'$ where $K'$ is an unramified
extension of $K$ and the norm of an element $(\alpha)\in\prod K'$
is the product of the individual norms of the components with respect
to the unramified extension $K'/K$. On the other hand, the image
of the norm map $K'^{*}\rightarrow K^{*}$ certainly contains $\mathcal{O}_{K}^{*}$
as $K'/K$ is unramified.
\end{proof}
\begin{cor}
\label{Corollary:rank_1_slope_0_descends}Let $(V,F)$ be a rank-1
$F$-isocrystal on $\mathbb{F}_{q}$ with coefficients in $L$. If
the eigenvalue of $F^{d}$ lives in finite index subfield $K\subset L$
and $(V,F)$ has slope 0, then $(V,F)$ descends to $K$.
\end{cor}

\begin{proof}
Apply Corollary \ref{Corollary:Rank_1_Descent} and Lemma \ref{Lemma:units_are_norms_unramified}.
\end{proof}
\begin{rem}
The slope 0 subcategory part of $\textbf{F-Isoc}(k)$ is a neutral
Tannakian category by \cite[Theorem 2.1]{crew1985f}. Therefore one
may alternatively use Proposition \ref{Proposition:Neutral_Rank_1_Descent}
to prove Corollary \ref{Corollary:rank_1_slope_0_descends}.
\end{rem}

It is an easy exercise to check that the irreducible objects of $\textbf{F-Isoc}(\mathbb{F}_{q})_{\overline{\mathbb{Q}}_{p}}$
have rank 1. Let $L$ be a $p$-adic local field. As semi-simple and
absolutely semi-simple objects of $\textbf{F-Isoc}(\mathbb{F}_{q})_{L}$
coincide by \cite[Lemme 5.2]{deligne2014semi}, it follows from Proposition
\ref{Proposition:Rank_1_Eigenvalue} that a semi-simple object of
$\textbf{F-Isoc}(\mathbb{F}_{q})_{L}$ is determined up to isomorphism
by $P_{F}(t)$.

\begin{prop}
\label{Proposition: Twisting_Charpoly}Let $L/K$ be a finite Galois
extension with group $G$ with $K$ a $p$-adic local field. Let $(V,F)$
be an $F$-isocrystal on $\mathbb{F}_{q}$with coefficients in $L$
and denote $(^{g}V,^{g}F):=\ ^{g}(V,F)$. Then 
\[
P_{^{g}F}(t)=\ ^{g}P_{F}(t)
\]
\end{prop}

\begin{proof}
First of all, $V$ is a finite free $\mathbb{Q}_{q}\otimes L$ module.
Let $g\in G$ and consider the object $^{g}(V,F)$. The underlying
sets $V$ and $^{g}V$ may be naturally identified, and if $v\in V$,
we write $^{g}v$ for the corresponding element of $^{g}V$; here
$l.(^{g}v)=^{g}(g^{-1}(l).v)$. Pick a free basis $\{v_{i}\}$ of
$V$ and let the ``matrix'' of $F$ in this basis be $S$. Then
the ``matrix'' of $^{g}F$ in the basis $\{^{g}v_{i}\}$ is $^{g}S$.
Moreover, the actions of $G$ and $\sigma$ commute, so $(^{g}F)^{d}=\ ^{g}(F^{d})$.
Therefore, $P_{^{g}F}(t)=\ ^{g}P_{F}(t)$ as desired.
\end{proof}
\begin{cor}
Let $L/L_{0}$ be a Galois extension of $p$-adic local fields with
group $G$. Let $E=(V,F)$ be a semi-simple object of $\textbf{F-Isoc}(\mathbb{F}_{q})_{L}$.
Then $P_{F}(t)\in L_{0}[t]$ implies that $^{g}E\cong E$ for all
$g\in G$. 
\end{cor}

\begin{proof}
Immediate from Proposition \ref{Proposition: Twisting_Charpoly} and
the fact that a semi-simple $F$-isocrystal is determined up to isomorphism
by its characteristic polynomials.
\end{proof}
{} %
{} 

\section{\label{Section:Compatible_systems}Coefficient objects, Compatible
Systems and Companions}

For a more comprehensive introduction to this section, see \cite{kedlaya2016notes,kedlayacompanions}. 
\begin{defn}
\cite{kedlayacompanions} Let $X/\mathbb{F}_{q}$ be a smooth, geometrically
connected variety. The category $\textbf{Weil}(X)$ denotes the $\mathbb{Q}_{l}$-linear
Tannakian category of lisse Weil sheaves on $X$. A \emph{coefficient
object }is an object either of $\textbf{Weil}(X)_{\overline{\mathbb{Q}}_{l}}$
or $\textbf{F-Isoc}^{\dagger}(X)_{\overline{\mathbb{Q}}_{p}}$. We
informally call the former the \emph{étale case }and the latter the
\emph{crystalline case}. We say that a coefficient object has coefficients
in $K$ if may be descended to the appropriate category with coefficients
in $K$.
\end{defn}

Given a coefficient object $\mathcal{F}$ and a closed point $x$
of $X$, the notation $P_{x}(\mathcal{F},t)$ refers to the reverse
characteristic polynomial of Frobenius of $\mathcal{F}$ at $x$. 
\begin{defn}
Let $\mathcal{F}$ be a coefficient object. We say $\mathcal{F}$
is \emph{algebraic} if $P_{x}(\mathcal{F},t)\in\overline{\mathbb{Q}}[t]$
for all closed points $x\in|X|$. Let $E$ be a number field. We say
$\mathcal{F}$ is $E$\emph{-algebraic }if $P_{x}(\mathcal{F},t)\in E[t]$
for all closed points $x\in|X|$.

\label{Definition:Companion}Let $X$ be a normal geometrically connected
variety over $\mathbb{F}_{q}$. Let $\mathcal{E}$ and $\mathcal{E}'$
be semi-simple algebraic coefficient objects on $X$ with coefficient
fields $K$ and $K'$ respectively. Fix an isomorphism $\iota:\overline{K}\rightarrow\overline{K'}$.
We say $\mathcal{E}$ and $\mathcal{E}'$ are \emph{$\iota$-companions}
if $^{\iota}P_{x}(\mathcal{E},t)=P_{x}(\mathcal{E}',t)$ for all closed
points $x$ and $X$.

\begin{defn}
Let $X$ be a normal geometrically connected variety over $\mathbb{F}_{q}$
and let $E$ be a number field. Then an \emph{$E$-compatible system}
is a system of $E_{\lambda}$-coefficient objects $(\mathcal{E}_{\lambda})_{\lambda\nmid p}$
over places $\lambda\nmid p$ of $E$ such that for each closed point
$x$ of $X$, we have
\[
P_{x}(\mathcal{E}_{\lambda},t)\in E[t]\subset E_{\lambda}[t]
\]
and furthermore $P_{x}(\mathcal{E}_{\lambda},t)$ is independent of
$\lambda$. A \emph{complete $E$-compatible system} $(\mathcal{E}_{\lambda})$
is an $E$-compatible system together with, for each $\lambda|p$,
an object 
\[
\mathcal{E}_{\lambda}\in\textbf{F-Isoc}^{\dagger}(X)_{E_{\lambda}}
\]
such that for every place $\lambda$ of $E$ and for every closed
point $x$ of $X$, $P_{x}(\mathcal{E}_{\lambda},t)\in E[t]\subset E_{\lambda}[t]$
is independent of $\lambda$ and that $(\mathcal{E}_{\lambda})$ satisfies
the following completeness condition:
\end{defn}

\begin{condition*}
Consider $\mathcal{E}_{\lambda}$ as a $\overline{\mathbb{Q}}_{\lambda}$-coefficient
object. Then for any $\iota:\overline{\mathbb{Q}}_{\lambda}\rightarrow\overline{\mathbb{Q}}_{\lambda'}$,
the $\iota$-companion to $\mathcal{E}_{\lambda}$ is isomorphic to
an element of $(\mathcal{E}_{\lambda})$.
\end{condition*}
\end{defn}

\begin{rem}
The $\iota$ in Definition \ref{Definition:Companion} does not reflect
the topology of $\overline{\mathbb{Q}}_{p}$ or $\overline{\mathbb{Q}}_{l}$;
in particular, it need not be continuous. The completeness condition
demands that ``all possible companions exist''. Implicit in the
definition is the result, due to Abe, Deligne, Esnault, Lafforgue,
and Kedlaya that, for a given irreducible coefficient object $\mathcal{E}$
and fixed prime $\lambda$, there are only finitely many $\overline{\mathbb{Q}}_{\lambda}$-companions
to $\mathcal{E}$.
\end{rem}

{} %
{} 

The following fundamental result follows from the work on the Langlands
correspondence due to Lafforgue and Abe \cite{lafforgue2002chtoucas,abe2013langlands,abe2011langlands}
\begin{thm}
\label{Theorem:Abe_correspondence}(Abe, Lafforgue) Let $C$ be a
smooth curve over $\mathbb{F}_{q}$. 
\begin{enumerate}
\item Deligne's companions conjecture \cite[1.2.10]{deligne1980conjecture}
is true for $C$.
\item Let $l\neq p$ be a prime. For every isomorphism $\iota:\overline{\mathbb{Q}}_{l}\rightarrow\overline{\mathbb{Q}}_{p}$,
there is a bijective correspondence 
\[
\left\{ \begin{alignedat}{1} & \text{Local systems }\mathscr{L}\text{ on }C\text{ such that}\\
 & \bullet\text{\ensuremath{\mathscr{L}}}\text{ has coefficients in }\overline{\mathbb{Q}}_{l}\\
 & \bullet\mathscr{L}\text{ is irreducible of rank }n\\
 & \bullet\mathscr{L}\text{ has finite determinant}\\
 & \text{up to isomorphism}
\end{alignedat}
\right\} \longleftrightarrow\left\{ \begin{aligned} & \text{Overconvergent \ensuremath{F}-Isocrystals }\mathscr{E}\text{ on \ensuremath{C} such that}\\
 & \bullet\mathscr{E}\text{ has coefficients in }\overline{\mathbb{Q}}_{p}\\
 & \bullet\mathscr{E}\text{ is irreducible of rank }n\\
 & \bullet\mathscr{E}\text{ has finite determinant}\\
 & \text{up to isomorphism}
\end{aligned}
\right\} 
\]
depending on $\iota$ such that $\mathscr{L}$ and $\mathscr{E}$
are $\iota$-compatible.
\item Let $\mathcal{E}$ be an irreducible coefficient object with finite
order determinant. Then there exists a number field $E$ such that
$\mathcal{E}$ is part of a complete $E$-compatible system.
\end{enumerate}
\end{thm}

\section{\label{Section:F_isocrystals_local_systems_on_curves}$F$-isocrystals
and $l$-adic local systems on Curves}

Let $C/\mathbb{F}_{q}$ be a smooth connected curve and let $\overline{c}$
be a geometric point. Let $\mathscr{L}$ be an irreducible $l$-adic
local system on $C$ with trivial determinant, which one may think
of as a continuous representation $\rho_{l}:\pi_{1}(C,\overline{c})\rightarrow SL(n,\overline{\mathbb{Q}}_{l})$.
Theorem \ref{Theorem:Abe_correspondence} tells us that we can find
a number field $E$ such that $\rho_{l}$ fits into a complete $E$-compatible
system.
\begin{example}
\emph{\label{Example:Fake_elliptic_curves_brauer}The number field
$E$ can in general be larger than the field extension of $\mathbb{Q}$
generated by the coefficients of the characteristic polynomials of
all Frobenius elements $F_{x}$.} For instance, let $D$ be a non-split
quaternion algebra over $\mathbb{Q}$ that is split at $\infty$ and
let $p$ be a prime where $D$ splits. The Shimura curve $X^{D}$
exists as a smooth complete stacky curve over $\mathbb{F}_{p}$ and
it admits a universal abelian surface $f:\mathcal{A}\rightarrow X^{D}$
with multiplication by $\mathcal{O}_{D}$. It is an exercise to check
that if $l$ splits if and only if $R^{1}f_{*}\mathbb{Q}_{l}$ splits
as $\mathscr{L}\oplus\mathscr{L}$. In particular, even though \emph{all
Frobenius traces are in $\mathbb{Q}$, not all of the $l$-adic companions
can be defined over $\mathbb{Q}_{l}$.} In other words, they do not
form a $\mathbb{Q}$-compatible system. %
{} 
\end{example}

Using our simple descent machinery, we construct criteria for the
field-of-definition of a coefficient object to be as small as possible.
\begin{lem}
\label{Lemma:Main_Isocrystal_Descent}Let $L/\mathbb{Q}_{p}$ be a
finite extension, $C$ a smooth curve over $\mathbb{F}_{q}$, and
$\mathscr{E}$ an absolutely irreducible overconvergent $F$-isocrystal
on $C$ of rank $n$ with coefficients in $L$. Suppose that there
exists a closed point $i:x\rightarrow C$ such that $i^{*}\mathscr{E}$
has 0 as a slope with multiplicity 1. Suppose further that for all
closed points $x\in|C|$, $P_{\mathscr{E}}(x,t)\in L_{0}[t]$ for
some $p$-adic subfield $L_{0}\subset L$. Then $\mathscr{E}$ can
be descended to an $F$-isocrystal with coefficients in $L_{0}$.
\end{lem}

\begin{proof}
By enlarging $L$ we may suppose that $L/L_{0}$ is Galois; let $G=\text{Gal}(L/L_{0})$.
As $\mathscr{E}$ is irreducible, for any $g\in G$, $^{g}\mathscr{E}\cong\mathscr{E}$
by Proposition \ref{Proposition: Twisting_Charpoly} and \cite[A.4.1]{abe2013langlands}.
Now let us use Lemma \ref{Lemma:Filtered_Descent} for the restriction
functor
\[
i^{*}:\textbf{F-Isoc}^{\dagger}(C)\rightarrow\textbf{F-Isoc}(x)
\]
Because 0 occurs as a slope with multiplicity 1 in $i^{*}\mathscr{E}$
we can write $i^{*}\mathscr{E}\cong N_{1}\oplus N_{2}$ under the
isoclinic decomposition. Here $N_{1}$ has rank 1 and unique slope
0 while no slope of $N_{2}$ is 0. There are no maps between $N_{1}$
and any Galois twist of $N_{2}$ because twisting an $F$-isocrystal
does not change the slope. The endomorphism algebra of any rank-1
object in $F\mbox{\text{-Isoc}}(\mathbb{F}_{q})_{L}$ is $L$. Now,
$\text{End}(\mathscr{E})\cong L$ by Schur's Lemma because $\mathscr{E}$
is absolutely irreducible.

Finally, we must argue that $N_{1}$ descends to $L_{0}$. Let $x=\text{Spec}(\mathbb{F}_{p^{d}})$.
The fact that the slope 0 occurs exactly once in $i^{*}\mathscr{E}$
implies that the eigenvalue $\alpha$ of $F^{d}$ on $N_{1}$ is an
element of $L_{0}$. Moreover, $\alpha\in\mathcal{O}_{L_{0}}^{*}$
because $N_{1}$ is slope 0. By Corollary \ref{Corollary:rank_1_slope_0_descends},
$N_{1}$ descends to $L_{0}$. The hypotheses of Lemma \ref{Lemma:Filtered_Descent}
are all satisfied and we may conclude that $\mathscr{E}$ (and $N_{2}$)
descends to $L_{0}$.
\end{proof}
We remark that Lemma \ref{Lemma:Main_Isocrystal_Descent} is the $p$-adic
analog of a proposition of Chin \cite[Proposition 7]{chin2003independence}.
The following is a special case of \cite[Theorem 1.4]{koshikawa2015overconvergent}.
For completeness, we give a short proof.

\begin{lem}
\label{Lemma:supersingular_finite_image}Let $C$ be a smooth curve
over $\mathbb{F}_{q}$, and $\mathscr{E}$ an irreducible rank $n$
overconvergent $F$-isocrystal on $C$ with coefficients in $\overline{\mathbb{Q}}_{p}$
such that $\mathscr{E}$ has trivial determinant. By Theorem \ref{Theorem:Abe_correspondence}
there is a number field $E$ such that $\mathscr{E}$ is part of a
complete $E$-compatible system. In particular, for every $\lambda|p$
there is a compatible overconvergent $F$-isocrystal $\mathscr{E}_{\lambda}$
with coefficients in $E_{\lambda}$.

Suppose that for all $\lambda|p$ and for all closed points $x\in|C|$,
$i_{x}^{*}\mathscr{E}_{\lambda}$ is an isoclinic $F$-isocrystal
on $x$. Then the representation has finite image: for instance, for
every $\lambda\nmid p$, the associated $\lambda$-adic representation
has finite image. Equivalently, the ``motive'' can be trivialized
by a finite étale cover $C'\rightarrow C$
\end{lem}

\begin{proof}
The eigenvalues of $F_{x}$ are $\lambda$-adic units for all $\lambda\nmid p$
by \cite[Theorem 0.2.1]{kedlayacompanions}. On the other hand, for
each $\lambda|p$ and for every closed point $x\in|C|$, $i_{x}^{*}\mathscr{E}_{\lambda}$
being isoclinic and having trivial determinant implies the slopes
of $i_{x}^{*}\mathscr{E}_{\lambda}$ are 0 and hence that the eigenvalues
of $F_{x}$ are $\lambda$-adic units. As eigenvalues of $F_{x}$
are algebraic numbers, this implies that they are all roots of unity.
Moreover, each of these roots of unity lives in a degree $n$ extension
of $E$ and there are only finitely many roots of unity that live
in such extensions: there are only finitely many roots of unity with
fixed bounded degree over $\mathbb{Q}$. Therefore there are only
finitely many eigenvalues of $F_{x}$ as $x$ ranges through the closed
points of $C$.

Now, pick $\lambda\nmid p$ and consider the associated representation
$\rho_{\lambda}:\pi_{1}(C,\overline{c})\rightarrow SL(n,E_{\lambda})$.
By the above discussion, there exists some integer $k$ such that
for every closed point $x\in|C|$, the generalized eigenvalues $\rho_{\lambda}(F_{x})$
are $k$\textsuperscript{th} roots of unity. But Frobenius elements
are dense and $\rho_{\lambda}$ is a continuous homomorphism, so that
the same is true for the entire image of $\rho_{\lambda}$. The image
of $\rho_{\lambda}$ therefore only has finitely many traces.

Burnside proved that if $G\subset GL(n,\mathbb{C})$ has finitely
many traces and the associated representation is irreducible, then
$G$ is finite: see, for instance%
{} \cite[19.A.9]{rowen2008graduate}. Thus the entire image of $\rho_{\lambda}$
is finite.
\end{proof}
\begin{prop}
\label{Proposition:Rank_2_twisted_descent}Let $C$ be a smooth curve
over $\mathbb{F}_{q}$ and let $\mathscr{L}$ be an absolutely irreducible
rank 2 $l$-adic local system with trivial determinant, infinite image,
and all Frobenius traces in a number field $E$. There exists a finite
extension $F$ of $E$ such that $\mathscr{L}$ is a part of a complete
$F$-compatible system. Not every crystalline companion $\mathscr{E}_{\lambda}\in\textbf{F-Isoc}(C)_{F_{\lambda}}$
is everywhere isoclinic by Lemma \ref{Lemma:supersingular_finite_image}.
For each such not-everywhere-isoclinic $\mathscr{E}_{\lambda}$ there
exists 

\begin{enumerate}
\item A positive rational number $\frac{s}{r}$
\item A finite field extension $\mathbb{F}_{q'}/\mathbb{F}_{q}$ with $C'$
denoting the base change of $C$ to $\mathbb{F}_{q'}$ ($q'$ will
be the least power of $q$ divisible by $p^{r}$)
\end{enumerate}
such that the $F$-isocrystal $\mathscr{M}:=\mathscr{E}_{\lambda}\otimes\overline{\mathbb{Q}}_{p}(-\frac{s}{r})$
on $C'$ descends to $E_{\lambda}$.

\end{prop}

\begin{proof}
Theorem \ref{Theorem:Abe_correspondence} imply that there is a finite
extension $F/E$, without loss of generality Galois, such that $\mathscr{L}$
fits into a complete $F$-compatible system on $C$. As we assumed
$\mathscr{L}$ had infinite image, Lemma \ref{Lemma:supersingular_finite_image}
implies that there is a place $\lambda|p$ of $F$ together with an
object $\mathscr{E}_{\lambda}\in\textbf{F-Isoc}^{\dagger}(C)_{F_{\lambda}}$
that is compatible with $\mathscr{L}$ and such that the general point
is not isoclinic. We abuse notation and denote the restriction of
$\lambda$ to $E$ by $\lambda$ again. The object $\mathscr{E}_{\lambda}$
is isomorphic to its twists by $\text{Gal}(F_{\lambda}/E_{\lambda})$
by \cite[A.4.1]{abe2013langlands} because $P_{\mathscr{E}_{\lambda}}(x,t)\in E[t]$
for all closed points $x\in X$. Pick a closed point $i:x\rightarrow C$
such that $i^{*}\mathscr{E}_{\lambda}$ has slopes $(-\frac{s}{r},\frac{s}{r})$.
Let $q'$ be the smallest power of $q$ that is divisible by $p^{r}$
and let $C'$ denote the base change of $C$ to $\mathbb{F}_{q'}$.

Consider the twist $\mathscr{M}:=\mathscr{E}_{\lambda}\otimes\overline{\mathbb{Q}}_{p}(-\frac{s}{r})$,
thought of as an $F$-isocrystal on $C'$ with coefficients in $\overline{\mathbb{Q}}_{p}$.
We have enlarged $q$ to $q'$ where $p^{r}|q'$; Remark \ref{Remark:Twists_of_fractional_tate_motives}
says that $\overline{\mathbb{Q}}_{p}(-\frac{s}{r})$, considered as
object of $\textbf{F-Isoc}(\mathbb{F}_{q'})_{\overline{\mathbb{Q}}_{p}}$,
is isomorphic to all of its twists by $\text{Gal}(\overline{\mathbb{Q}}_{p}/\mathbb{Q}_{p})$.
Therefore, $P_{\mathscr{M}}(x,t)\in E_{\lambda}[t]$ for all closed
points $x$. It again follows from \cite[A.4.1]{abe2013langlands}
that $\mathscr{M}$ is isomorphic to all of its Galois twists by $\text{Gal}(\overline{\mathbb{Q}}_{p}/E_{\lambda})$.
At the point ``$x$'', the slopes are now $(0,\frac{2s}{r})$. Apply
Lemma \ref{Lemma:Main_Isocrystal_Descent} to descend $\mathscr{M}$
to the field of traces $E_{\lambda}$, as desired.
\end{proof}
\begin{rem}
The point of Proposition \ref{Proposition:Rank_2_twisted_descent}
is that \emph{the field of definition of $\mathscr{M}$ is $E_{\lambda}$,
a completion of the field of traces}. Proposition \ref{Proposition:Rank_2_twisted_descent}
has the following slogan: for every not-everywhere-isoclinic crystalline
companion of $\mathscr{L}$, there exists a twist such that the Brauer
obstruction vanishes. This is in contrast to the $l$-adic case, where
there is no \emph{à priori} reason an $l$-adic Brauer obstruction
should vanish.
\end{rem}

We now specialize to the case where $\mathscr{L}$ is a rank 2 $l$-adic
local system with trivial determinant, infinite image, and having
all Frobenius traces in a number field $E$ where $p$ splits completely.
Proposition \ref{Proposition:Rank_2_twisted_descent} implies that,
up to extension of the ground field $\mathbb{F}_{q}$, we can find
an $F$-isocrystal $\mathscr{E}$ with \emph{coefficients in $\mathbb{Q}_{p}$}
that is compatible with $\mathscr{L}$ up to a twist and is not everywhere
isoclinic. Moreover, by construction, there is a point $x$ such that
the slopes of $\mathscr{E}_{x}$ are $(0,\frac{2s}{r})$. On the one
hand the slope of the determinant of $\mathscr{E}_{x}$ is necessarily
an integer because the coefficients of $\mathscr{E}$ are $\mathbb{Q}_{p}$.
Therefore $\frac{2s}{r}$ is a positive integer. On the other hand,
\cite[Corollaire 2.2]{lafforgue2011estimees} implies that $\frac{2s}{r}\leq1$,
so $\frac{s}{r}=\frac{1}{2}$ and $\text{det}(\mathscr{E})\cong\mathbb{Q}_{p}(-1)$.
We record this analysis in the following important corollary.
\begin{cor}
\label{Corollary:rank2_p-adic_companion_Qp}Let $C$ be a curve over
$\mathbb{F}_{q}$ and let $\mathscr{L}$ be an absolutely irreducible
rank 2 $l$-adic local system with trivial determinant, infinite image,
and all Frobenius traces in $\mathbb{Q}$. Suppose $q$ is a square.
Then there exists a unique absolutely irreducible overconvergent $F$-isocrystal
$\mathscr{E}$ \emph{with coefficients in $\mathbb{Q}_{p}$ }that
is compatible with $\mathscr{L}\otimes\mathbb{Q}_{l}(-\frac{1}{2})$.
By construction, $\mathscr{E}$ is generically ordinary i.e. there
exists a closed point $x\in|C|$ such that $\mathscr{E}_{x}$ has
slopes $(0,1)$.
\end{cor}

\begin{proof}
Only the uniqueness needs to be proved. As $q$ is a square, the character
$\overline{\mathbb{Q}}_{l}(\frac{1}{2})$ in fact descends to a character
$\mathbb{Q}_{l}(\frac{1}{2})$ (because $q$ is a quadratic residue
mod $l$) and the coefficients of the characteristic polynomials of
the Frobenius elements on $\mathscr{L}\otimes\mathbb{Q}_{l}(-\frac{1}{2})$
are all in $\mathbb{Q}$. By \cite[A.4.1]{abe2013langlands}, any
absolutely irreducible $F$-isocrystal on $C$ is uniquely determined
by these characteristic polynomials as there is a unique embedding
$\mathbb{Q}\hookrightarrow\overline{\mathbb{Q}}_{p}$.
\end{proof}
\begin{lem}
\label{Lemma:not_everywhere_ordinary}Let $C$ a complete curve over
a perfect field $k$, and $\mathscr{E}$ an absolutely irreducible
rank-2 $F$-isocrystal on $C$ with coefficients in $\mathbb{Q}_{p}$.
Suppose $\mathscr{E}$ has determinant $\mathbb{Q}_{p}(-1)$. Then
there is a closed point $x\in|C|$ such that $\mathscr{E}_{x}$ is
isoclinic.
\end{lem}

\begin{proof}
As we are assuming $\mathscr{E}$ is rank 2, has coefficients in $\mathbb{Q}_{p}$,
and determinant $\mathbb{Q}_{p}(-1)$, if either of the slopes of
$\mathscr{E}_{x}$ were non-integral, the slopes would have to be
$(\frac{1}{2},\frac{1}{2})$. If the slopes were integral, they are
integers $(a,b)$ that sum to 1. On other other hand, \cite[Corollaire 2.2]{lafforgue2011estimees}
shows that in this case, the slopes must be $(0,1)$. Assume that
$\mathscr{E}_{x}$ is never isoclinic. Then for every closed point
$x$ of $|C|$, $\mathscr{E}_{x}$ has slopes $(0,1)$. The slope
filtration \cite[2.6.2]{katz1979slope} is therefore non-trivial,
which contradicts the irreducibility of $\mathscr{E}$.
\end{proof}

Combining Lemma \ref{Lemma:not_everywhere_ordinary} with the $\mathbb{Q}_{p}$
companion $\mathscr{E}$ constructed in Corollary \ref{Corollary:rank2_p-adic_companion_Qp},
we see that $\mathscr{E}$ is generically ordinary and has (finitely
many) supersingular points. In particular, the slopes of $\mathscr{E}_{x}$
for $x\in|C|$ are either $(0,1)$ or $(\frac{1}{2},\frac{1}{2})$.

\begin{cor}
\label{Corollary:local_system_p-div}Let $C$ be a complete curve
over $\mathbb{F}_{q}$ and let $\mathscr{L}$ be an absolutely irreducible
rank 2 $l$-adic local system with trivial determinant, infinite image,
and all Frobenius traces in $\mathbb{Q}$. There is a BT group $\mathscr{G}$
on $C$ with the following properties.

\begin{itemize}
\item $\mathscr{G}$ has height 2 and dimension 1.
\item $\mathscr{G}$ has slopes $(0,1)$ and $(\frac{1}{2},\frac{1}{2})$
(i.e., $\mathscr{G}$ is generically ordinary with supersingular points).
\item The Dieudonné crystal $\mathbb{D}(\mathscr{G})$ is compatible with
$\mathscr{L}(-\frac{1}{2})$.
\end{itemize}
$\mathscr{G}$ has the following weak uniqueness property: the $F$-isocrystal
$\mathbb{D}(\mathscr{G})$ is unique.
\end{cor}

\begin{proof}
First of all, we have constructed an absolutely irreducible $\mathscr{E}\in\textbf{F-Isoc}(C)$
with determinant $\mathbb{Q}_{p}(-1)$ so that $\mathscr{E}$ is compatible
with $\mathscr{L}(-1/2)$ and the slopes of $\mathscr{E}$ are in
the interval $[0,1]$ in Corollary \ref{Corollary:rank2_p-adic_companion_Qp}.
It follows from \cite[Lemma 5.8]{kp2018} that $\mathscr{E}$ underlies
a (non-unique) Dieudonné crystal $\mathscr{M}$. (This argument is
essentially due to Katz, see \cite[Theorem 2.6.1]{katz1979slope}.
For another argument, due to Crew, see \cite[Remark 2.3]{kedlaya2016notes}.)
Using \cite[Main Theorem]{de1995crystalline}, there exists a BT group
$\mathscr{G}$ on $C$ so that $\mathbb{D}(\mathscr{G})\cong\mathscr{M}$.
Moreover, by Lemma \ref{Lemma:not_everywhere_ordinary}, it follows
that $\mathscr{E}$ and hence $\mathscr{G}$ has both ordinary and
supersingular points. 

The weak uniqueness comes from the following facts: an absolutely
irreducible overconvergent $F$-isocrystal with trivial determinant
is uniquely specified by Frobenius eigenvalues, $\mathbb{D}(\mathscr{G})$
is absolutely irreducible because it has both ordinary and supersingular
points, and there is a unique embedding $\mathbb{Q}\hookrightarrow\overline{\mathbb{Q}}_{p}$.
\end{proof}
We emphasize that the BT group $\mathscr{G}$ constructed in Corollary
\ref{Corollary:local_system_p-div} is not unique.

\section{\label{Section:Kodaira_spencer}Kodaira-Spencer}

This section discusses the deformation theory of Barsotti-Tate groups
in order to refine Corollary \ref{Corollary:local_system_p-div}.
It will also be useful in the applications to Shimura curves. Throughout
this section, we use the terms ``Barsotti-Tate group'', ``BT group'',
and ``$p$-divisible group'' interchangeably. The main references
are Illusie \cite{illusie1985deformations}, Xia \cite{xia2013deformation},
and de Jong \cite[Section 2.5]{de1995crystalline}.

Let $S$ be a smooth scheme over a perfect field $k$ of characteristic
$p$. Given a BT group $\mathscr{G}$ over $S$ and a closed point
$x\in|S|$, there is a map of formal schemes $u_{x}:S_{/x}^{\mathcircumflex}\rightarrow Def(\mathscr{G}_{x})$
to the (equal characteristic) universal deformation space of $\mathscr{G}_{x}$.
If $\mathscr{G}$ has dimension $d$ and codimension $c$, then the
dimension of the universal deformation space is $cd$ \cite[Corollary 4.8 (i)]{illusie1985deformations}.
In particular, in the case $\mathscr{G}$ has height 2 and dimension
1, $Def(\mathscr{G}_{x})$ is a one-dimensional formal scheme, exactly
as in the familiar elliptic modular case.

Let $\mathscr{G}\rightarrow S$ be a Barsotti-Tate group. Then $\mathbb{D}(\mathscr{G})$
is a Dieudonné crystal which we may evaluate on $S$ to obtain a vector
bundle $\mathbb{D}(\mathscr{G})_{S}$ of rank $c+d=\text{ht}(\mathscr{G})$.
Let be $\omega$ the \emph{Hodge bundle }of $\mathscr{G}$ (equivalently,
of $\mathscr{G}[p]$): if $e:S\rightarrow\mathscr{G}[p]$ is the identity
section, then $\omega:=e^{*}\Omega_{\mathscr{G}[p]/S}^{1}$. Let $\alpha$
be the dual of the Hodge bundle of the Serre dual $\mathscr{G}^{t}$.
There is the Hodge filtration
\[
0\rightarrow\omega\rightarrow\mathbb{D}(\mathscr{G})_{S}\rightarrow\alpha\rightarrow0
\]
We remark that $\omega$ has rank $d$ and $\alpha$ has rank $c$.
The Kodaira-Spencer map is as follows
\[
KS:T_{S}\rightarrow Hom_{\mathcal{O}_{S}}(\omega,\alpha)\cong\omega^{*}\otimes\alpha
\]

Illusie \cite[A.2.3.6]{illusie1985deformations} proves that $KS_{x}$
is surjective if and only if $u_{x}$ is formally smooth, see also
\cite[2.5.5]{de1995crystalline}. This motivates the following definition.
\begin{defn}
Let $S/k$ be a smooth scheme over a perfect field $k$ of characteristic
$p$. Let $\mathscr{G}\rightarrow S$ be a Barsotti-Tate group and
$x\in|X|$ a closed point. We say that $\mathscr{G}$ is \emph{versally
deformed at} $x$ if either of the following equivalent conditions
hold
\begin{itemize}
\item The fiber at $x$ of $KS$, $KS_{s}:T_{S,x}\rightarrow(\omega^{*}\otimes\alpha)_{x}$,
is surjective.
\item The map $u_{x}:S_{/x}^{\mathcircumflex}\rightarrow Def(\mathscr{G}_{x})$
is formally smooth.
\end{itemize}
\end{defn}

\begin{defn}
Let $S/k$ be a smooth scheme over a perfect field $k$ of characteristic
$p$. Let $\mathscr{G}\rightarrow S$ be a Barsotti-Tate group. We
say that $\mathscr{G}$ is \emph{generically versally deformed }if,
for every connected component $S_{i}$, there exists a closed point
$x_{i}\in|S_{i}|$ such that $\mathscr{G}$ is versally deformed at
$x_{i}$.
\end{defn}

\begin{rem}
If $\mathscr{G}\rightarrow S$ is generically versally deformed, there
exists a dense Zariski open $U\subset X$ such that $\mathscr{G}|_{U}\rightarrow U$
is \emph{everywhere versally deformed. }This is because the condition
``a map of vector bundles is surjective'' is an open condition.
\end{rem}

\begin{example}
\label{Example:modular_igusa_not_everywhere_versally_deformed}Let
us recall Igusa level structures, as in \cite{ulmer1990universal}
. Let $Y(1)=\mathcal{M}_{1,1}$. There is a universal elliptic curve
$\mathcal{E}\rightarrow Y(1)$. Let $\mathscr{G}=\mathcal{E}[p^{\infty}]$
be the associated $p$-divisible group over $Y(1)$. Here, $\mathscr{G}$
is height 2, dimension 1, and \emph{everywhere versally deformed}
on $Y(1)$. Let $X$ be the cover of $Y(1)$ that trivializes the
finite flat group scheme $\mathscr{G}[p]^{\acute{e}t}$ away from
the supersingular locus of $Y(1)$. $X$ is branched exactly at the
supersingular points. Pulling back $\mathscr{G}$ to $X$ yields a
BT group that is generically versally deformed but \emph{not everywhere
versally deformed} on $X$.
\end{example}

\begin{example}
Let $S=\mathcal{A}_{2,1}\otimes\mathbb{F}_{p}$, the moduli of principally
polarized abelian surfaces. Then the BT group of the universal abelian
scheme, $\mathscr{G}\rightarrow S$, is \emph{nowhere versally deformed}:
the formal deformation space of a height 4 dimension 2 BT group is
4, whereas $\dim(S)=3$. Serre-Tate theory relates this to the fact
that most formal deformations of an abelian variety of dimension $\dim(A)>1$
are not algebraizable.
\end{example}

\begin{lem}
\label{Lemma:xia_frob_untwist}(Xia's Frobenius Untwisting Lemma)
Let $S$ be a smooth scheme over a perfect field $k$ of characteristic
$p$. Let $\mathscr{G}\rightarrow S$ be a BT group of height 2 and
dimension 1. Then $\mathscr{G}\rightarrow S$ has trivial KS map if
and only if there exists a BT group $\mathscr{H}$ on $S$ such that
$\mathscr{H}^{(p)}\cong\mathscr{G}$.
\end{lem}

\begin{proof}
This is \cite[Theorem 6.1]{xia2013deformation}.
\end{proof}
Xia's Lemma \ref{Lemma:xia_frob_untwist} allows us to set up the
following useful equivalence for characterizing when a height 2 dimension
1 BT group $\mathscr{G}$ is generically versally deformed.
\begin{lem}
\label{Lemma:equivalence_not_versally_deformed}Let $C$ be a smooth
curve over a perfect field $k$ of characteristic $p$. Let $\mathscr{G}$
be a height 2, dimension 1 BT group over $C$. Let $\eta$ be the
generic point of $C$. Suppose $\mathscr{G}_{\eta}$ is ordinary.
Then the following are equivalent

\begin{enumerate}
\item The KS map is 0.
\item There exists a BT group $\mathscr{H}$ on $C$ such that $\mathscr{H}^{(p)}\cong\mathscr{G}$.
\item There exists a finite flat subgroup scheme $N\subset\mathscr{G}$
over $C$ such that $N$ has order $p$ and is generically étale.
\item The connected-étale exact sequence $\mathscr{G}_{\eta}[p]^{\circ}\rightarrow\mathscr{G}_{\eta}[p]\rightarrow\mathscr{G}_{\eta}[p]^{\acute{e}t}$
over the generic point splits.
\end{enumerate}
\end{lem}

\begin{proof}
The equivalence of (1) and (2) is Lemma \ref{Lemma:xia_frob_untwist}.
Now, let us assume (2). Then there is a Verschiebug map 
\[
V_{\mathscr{H}}:\mathscr{G}\cong\mathscr{H}^{(p)}\rightarrow\mathscr{H}
\]
whose kernel is generically étale and has order $p$ because we assumed
$\mathscr{G}$ (and hence $\mathscr{H}$) were generically ordinary.
Therefore (3) is satisfied. Conversely, given (3), $N$ is $p$-torsion
\cite{tate1970group}. Therefore we have a factorization:
\[
\xymatrix{\mathscr{G}\ar[rr]\ar[dr]_{\times p} &  & \mathscr{G}/N\\
 & \mathscr{G}\ar[ur]
}
\]
Set $\mathscr{H}=\mathscr{G}/N$. As we assumed $N$ was generically
étale and $\mathscr{G}$ was generically ordinary, the map $\mathscr{G\rightarrow}\mathscr{G}/N$
may be identified with Verschiebung: $\mathscr{H}^{(p)}\rightarrow\mathscr{H}$.
In particular, $\mathscr{H}^{(p)}\cong\mathscr{G}$ as desired.

Let us again assume (3). Then $N_{\eta}\subset\mathscr{G}_{\eta}[p]$
projects isomorphically onto $\mathscr{G}_{\eta}[p]^{\acute{e}t}$.
Therefore the connected-étale sequence over the generic point splits.
To prove the converse, simply take the Zariski closure of the section
to $\mathscr{G}_{\eta}[p]\rightarrow\mathscr{G}_{\eta}[p]^{\acute{e}t}$
inside of $\mathscr{G}[p]$ to get $N$ (this will be a finite flat
group scheme because $C$ is a smooth curve.)
\end{proof}
\begin{lem}
\label{Lemma:unique_generically_deformed_BT_group}Let $C$ be a smooth
curve over a perfect field $k$ of characteristic $p$. Suppose $\mathscr{G}$
and $\mathscr{G}'$ are height 2, dimension 1 BT groups over $C$
that are generically versally deformed and generically ordinary. Suppose
further that their Dieudonné isocrystals are isomorphic: $\mathbb{D}(\mathscr{G})\otimes\mathbb{Q}\cong\mathbb{D}(\mathscr{G}')\otimes\mathbb{Q}$.
Then $\mathscr{G}$ and $\mathscr{G}'$ are isomorphic.
\end{lem}

\begin{proof}
The isocrystals being isomorphic implies that there is an isogeny
$\mathbb{D}(\mathscr{G})\rightarrow\mathbb{D}(\mathscr{G}')$. Then
\cite[Main Theorem]{de1995crystalline} implies that their is an associated
isogeny $\phi:\mathscr{G}\rightarrow\mathscr{G}'$. By ``dividing
by $p$'', we may ensure that $\phi$ does not restrict to 0 on $\mathscr{G}[p]$.
Now suppose for contradiction that $\phi$ is not an isomorphism,
i.e. that it has a kernel. Then $\phi|_{\mathscr{G}[p]}$ also has
a nontrivial kernel. 

We have the following diagram of connected-generically étale sequences.
\[
\xymatrix{\mathscr{G}[p]^{\circ}\ar[d]\ar[r] & \mathscr{G}'[p]^{\circ}\ar[d]\\
\mathscr{G}[p]\ar[d]\ar[r] & \mathscr{G}'[p]\ar[d]\\
\mathscr{G}[p]^{\acute{e}t}\ar[r] & \mathscr{G}'[p]^{\acute{e}t}
}
\]
As we have assumed $\mathscr{G}$ is generically versally deformed,
the kernel of $\phi|_{\mathscr{G}[p]}$ cannot be generically étale
by (3) of Lemma \ref{Lemma:equivalence_not_versally_deformed}. Thus
the kernel must be the connected group scheme $\mathscr{G}[p]^{\circ}$
because the order of $\mathscr{G}[p]$ is $p^{2}$. We therefore get
a nonzero map $\mathscr{G}[p]^{et}\rightarrow\mathscr{G}'[p]$. Now
$\mathscr{G}[p]^{\acute{e}t}$ has order $p$ and is generically étale
by definition, so by (3) of Lemma \ref{Lemma:equivalence_not_versally_deformed},
$\mathscr{G'}$ is not generically versally deformed, which contradicts
our hypothesis.
\end{proof}
\begin{thm}
\label{Theorem:p-div_companion_versal-1} Let $C$ be a smooth, geometrically
irreducible, complete curve over $\mathbb{F}_{q}$. Suppose $q$ is
a square. There is a natural bijection between the following two sets.
\[
\left\{ \begin{alignedat}{1} & \overline{\mathbb{Q}}_{l}\text{-local systems }\mathscr{L}\text{ on }C\text{ such that}\\
 & \bullet\mathscr{L}\text{ is irreducible of rank 2}\\
 & \bullet\mathscr{L}\text{ has trivial determinant}\\
 & \bullet\text{The Frobenius traces are in }\mathbb{Q}\\
 & \bullet\mathscr{L}\text{ has infinite image,}\\
 & \text{up to isomorphism}
\end{alignedat}
\right\} \longleftrightarrow\left\{ \begin{aligned} & p\text{-divisible groups }\mbox{\ensuremath{\mathscr{G}}}\text{ on }C\text{ }\text{ such that }\\
 & \bullet\mathscr{G}\text{ has height 2 and dimension 1}\\
 & \bullet\mathscr{G}\text{ is generically versally deformed}\\
 & \bullet\mathscr{G}\text{ has all Frobenius traces in }\mathbb{Q}\\
 & \bullet\mathscr{G}\text{ has ordinary and supersingular points,}\\
 & \text{up to isomorphism}
\end{aligned}
\right\} 
\]
such that if $\mathscr{L}$ corresponds to $\mathcal{G}$, then $\mathscr{L}\otimes\mathbb{Q}_{l}(-1/2)$
is compatible with the $F$-isocrystal $\mathbb{D}(\mathscr{G})\otimes\mathbb{Q}$.
\end{thm}

\begin{proof}
Given such an $\mathscr{L}$, we can make a BT group $\mathscr{G}$
as in Corollary \ref{Corollary:local_system_p-div}. Xia's Lemma \ref{Lemma:xia_frob_untwist}
ensures that we can modify $\mathscr{G}$ to be generically versally
deformed by Frobenius ``untwisting''; this process terminates because
there are both supersingular and ordinary points, so the map to the
universal deformation space cannot be identically 0. This BT group
is unique up to (\emph{non-unique}) isomorphism by Lemma \ref{Lemma:unique_generically_deformed_BT_group}. 

To construct the map in the opposition direction, just reverse the
procedure. Given such a $\mathscr{G}$, first form Dieudonné isocrystal
$\mathbb{D}(\mathscr{G})\otimes\mathbb{Q}$. This is an absolutely
irreducible $F$-isocrystal because there are both ordinary and supersingular
points. Twisting by $\overline{\mathbb{Q}}_{p}(1/2)$ yields an irreducible
object of $\textbf{F-Isoc}^{\dagger}(X)_{\overline{\mathbb{Q}}_{p}}$
that has trivial determinant. By Theorem \ref{Theorem:Abe_correspondence},
for any $l\neq p$ there is a compatible $\mathscr{L}_{l}$ that is
absolutely irreducible and has all Frobenius traces in $\mathbb{Q}$.
This $\mathscr{L}_{l}$ is unique: there is only one embedding of
$\mathbb{Q}$ in $\overline{\mathbb{Q}}_{l}$ and $\mathscr{L}_{l}$
is uniquely determined by the Frobenius traces. Finally, $\mathscr{L}_{l}$
has infinite image: as $\mathscr{G}$ had both ordinary and supersingular
points, $\mathbb{D}(\mathscr{G})\otimes\mathbb{Q}$ cannot be trivialized
on a finite étale cover.
\end{proof}
\begin{rem}
Theorem \ref{Theorem:p-div_companion_versal-1} has the following
strange corollary. Let $\mathscr{L}$ and $C$ be as in the theorem.
Then there is a natural effective divisor on $C$ associated to $\mathscr{L}$:
the points where the associated $\mathscr{G}$ is not versally deformed,
together with their multiplicity. This divisor is trivial if and only
if $\mathscr{G}\rightarrow C$ is everywhere versally deformed. We
wonder if this has an interpretation on the level of cuspidal automorphic
representations.
\end{rem}

\section{\label{Section:Algebraization_and_finite_monodromy}Algebraization
and Finite Monodromy}

In Theorem \ref{Theorem:p-div_companion_versal-1}, the finiteness
of the number of such local systems (a theorem whose only known proof
goes through the Langlands correspondence) implies the finiteness
of such BT groups. In general, BT groups on varieties are far from
being algebraic: for instance, over $\mathbb{F}_{p}$ there are uncountably
many BT groups of height 2 and dimension 1 as one can see from Dieudonné
theory. However, here they are constructed rather indirectly from
a motive via the Langlands correspondence. All examples of such rank
2 local systems that we can construct involve abelian schemes and
we are very interested in the following question.
\begin{question}
\label{Question:algebraize_G}Let $X$ be a smooth projective variety
over $\mathbb{F}_{q}$ and let $\mathscr{G}\rightarrow X$ be a height
2, dimension 1 $p$-divisible group with ordinary and supersingular
points. Is there an embedding as follows, where $A\rightarrow X$
is an abelian scheme?
\[
\xymatrix{\mathscr{G}\ar[r]\ar[dr] & A\ar[d]\\
 & X
}
\]
\end{question}

\begin{rem}
We explain the hypotheses Question \ref{Question:algebraize_G}. That
$X$ is complete ensures that the convergent $F$-isocrystal $\mathbb{D}(\mathscr{G})\otimes\mathbb{Q}$
is automatically overconvergent. The existence of both ordinary and
supersingular points ensures that the Dieudonné isocrystal $\mathbb{D}(\mathscr{G})\otimes\mathbb{Q}$
is absolutely irreducible. That $\mathbb{D}(\mathscr{G})\otimes\mathbb{Q}$
is is absolutely irreducible with cyclotomic determinant ensures that
the Frobenius traces are algebraic numbers.
\begin{rem}
In light of \cite{kp2018}, it follows that Question \ref{Question:algebraize_G}
reduces to the question for a sufficiently ample curve $C\subset X$
together with the $p$-adic companions conjecture for $\mathbb{D}(\mathscr{G})\otimes\mathbb{Q}$.
In particular, if the field generated by Frobenius traces of $\mathbb{D}(\mathscr{G})$
is isomorphic to $\mathbb{Q}$, then Question \ref{Question:algebraize_G}
for $(X,\mathscr{G})$ reduces to the question for $(C,\mathscr{G})$,
where $C\subset X$ is a curve that is the smooth complete intersection
of smooth ample divisors.
\end{rem}

\end{rem}

To motivate the next conjecture, recall that modular curves are not
complete. On the other hand, the universal local systems on Shimura
curves parameterizing fake elliptic curves cannot all be defined over
$\mathbb{Q}_{l}$, as we saw in Example \ref{Example:Fake_elliptic_curves_brauer}.
In other words, they do not form a $\mathbb{Q}$-compatible system.
\begin{conjecture}
\label{Conjecture:complete_curve_finite_monodromy}Let $X$ be a complete
curve over $\mathbb{F}_{q}$. Suppose $\{\mathscr{L}_{l}\}_{l\neq p}$
is a $\mathbb{Q}$-compatible system of absolutely irreducible rank
2 local systems with trivial determinant. Then they have finite monodromy.
\end{conjecture}

In particular, if Conjecture \ref{Conjecture:rank_2_fake_elliptic_curve}
is true for $(X,\mathscr{L}_{l})$, then Conjecture \ref{Conjecture:complete_curve_finite_monodromy}
is also true: if the monodromy were not finite, then as in Example
\ref{Example:Fake_elliptic_curves_brauer} it would follow that there
exists an $l'$ so that the $\overline{\mathbb{Q}}_{l'}$-companion
to $\mathscr{L}_{l}$ cannot be defined over $\mathbb{Q}_{l}$, contradicting
the assumption. Therefore Conjecture \ref{Conjecture:complete_curve_finite_monodromy}
may be used to \emph{falsify} Conjecture \ref{Conjecture:rank_2_fake_elliptic_curve}.

Using our techniques we can prove the related Theorem \ref{Theorem:F-Isoc-Q-Finite-Monodromy}:
it is a straightforward application of \cite[Corollaire 2.2]{lafforgue2011estimees}
. Note that, in the context of Conjecture \ref{Conjecture:complete_curve_finite_monodromy},
the hypothesis of Theorem \ref{Theorem:F-Isoc-Q-Finite-Monodromy}
is stronger exactly one way: namely, one assumes that the $p$-adic
companion of $\{\mathscr{L}_{l}\}_{l\neq p}$ has coefficients in
$\mathbb{Q}_{p}$. Note also that Theorem \ref{Theorem:F-Isoc-Q-Finite-Monodromy}
does not assume that $X$ is complete. 
\begin{thm}
\label{Theorem:F-Isoc-Q-Finite-Monodromy}Let $X$ be a curve over
$\mathbb{F}_{q}$. Let $\mathcal{E}\in\textbf{F-Isoc}^{\dagger}(X)$
be an overconvergent $F$-isocrystal on $X$ with coefficients in
$\mathbb{Q}_{p}$ that is rank 2, absolutely irreducible, and has
finite determinant. Suppose further that the field of traces of $\mathcal{E}$
is $\mathbb{Q}$. Then $\mathcal{E}$ has finite monodromy.
\end{thm}

\begin{proof}
We claim that $\mathscr{E}$ is isoclinic at every closed point $x\in|X|$.
Indeed, by \cite[Corollaire 2.2]{lafforgue2011estimees}, the slopes
of $\mathscr{E}_{x}$ differ by at most 1, forbidding slopes of the
form $(-a,a)$ for $0\neq a\in\mathbb{Z}$. As the coefficients of
$\mathscr{E}$ are $\mathbb{Q}_{p}$, any fractional slope must appear
more than once. As there is a \emph{unique }embedding $\mathbb{Q}\hookrightarrow\mathbb{Q}_{p}$,
this implies that any $p$-adic companion to $\mathscr{E}$ is isomorphic
to $\mathscr{E}$ itself by \cite[A.4.1]{abe2013langlands}. Therefore,
we may conclude by Lemma \ref{Lemma:supersingular_finite_image}.
\end{proof}
\begin{rem}
Note that Theorem \ref{Theorem:F-Isoc-Q-Finite-Monodromy} uses Lemma
\ref{Lemma:supersingular_finite_image} which critically uses \cite[Theorem 0.2.1]{kedlayacompanions},
i.e., a partial resolution to Deligne's companions conjecture. In
particular, we use that $\mathscr{E}$ lives in a complete compatible
system.
\end{rem}

\begin{cor}
Let $X$ be a curve over $\mathbb{F}_{q}$. Let $\mathscr{E}\in\textbf{F-Isoc}^{\dagger}(X)$
be an overconvergent $F$-isocrystal on $X$ (with coefficients in
$\mathbb{Q}_{p}$) that is rank 2, absolutely irreducible, and has
determinant $\mathbb{Q}_{p}(i)$ for $i\in2\mathbb{Z}$. Suppose further
that the field of traces of $\mathscr{E}$ is $\mathbb{Q}$. Then
$\mathscr{E}$ has finite monodromy.
\end{cor}

\begin{proof}
As $i$ is even, $\mathscr{E}(\frac{i}{2})\in\textbf{F-Isoc}^{\dagger}(X)$,
i.e. $\mathscr{E}(\frac{i}{2})$ has coefficients in $\mathbb{Q}_{p}$.
Apply Theorem \ref{Theorem:F-Isoc-Q-Finite-Monodromy}.
\end{proof}

\section{\label{Section:Shimura_Curves}Application to Shimura Curves}

In this section, we indicate a criterion for an étale correspondence
of projective hyperbolic curves over $\mathbb{F}_{q}$ to be the reduction
modulo $p$ of some Shimura curves over $\mathbb{C}$. Our goal was
to find a criterion that was as ``group theoretic'' as possible.
\begin{defn}
Let $X\leftarrow Z\rightarrow Y$ be a correspondence of smooth curves
over $k$. We say it \emph{has no core} if $k(X)\cap k(Y)$ has transcendence
degree 0 over $k$.
\end{defn}

For much more on the theory of correspondences without a core, see
our article \cite{krishnamoorthy2017correspondences}. In general,
correspondences of curves do not have cores. However, in the case
of étale correspondences over fields of characteristic 0, we have
the following remarkable result of Margulis, see \cite[Theorem 27]{margulis1991}
and \cite[Proposition 2.4]{mochizuki1998correspondences}.
\begin{thm}
(Margulis) If $X\leftarrow Z\rightarrow Y$ is a finite étale correspondence
of smooth hyperbolic curves without a core over a field $k$ of characteristic
0, then $X,Y,$ and $Z$ are all Shimura (arithmetic) curves. In particular,
all of the curves and maps can be defined over $\overline{\mathbb{Q}}$.
\end{thm}

\begin{rem}
Hecke correspondences of modular/Shimura curves furnish examples of
étale correspondences without a core.
\end{rem}

Our strategy will therefore be to find an additional structure on
an étale correspondence without a core such that the whole picture
\emph{canonically lifts} from $\mathbb{F}_{q}$ to characteristic
0 and apply Margulis' theorem. To make this strategy work, we need
two inputs. 
\begin{lem}
\label{Lemma:lift_no_core}Let $X\leftarrow Z\rightarrow Y$ be an
étale correspondence of projective hyperbolic curves over $\mathbb{F}$
without a core. If the correspondence lifts to $W(\mathbb{F})$, then
the lifted correspondence is étale and has no core. In particular,
$X$, $Z$, and $Y$ are the reductions modulo $p$ of Shimura curves.
\end{lem}

\begin{proof}
That the lifted correspondence is étale follows from open-ness of
the étale locus. It has no core by \cite[Lemma 4.14]{krishnamoorthy2017correspondences}.
Apply Margulis' theorem to the generic fiber.
\end{proof}
Xia proved the following theorem \cite[Theorem 1.2]{xia2013deformation},
which may be thought of as ``Serre-Tate canonical lift'' in for
hyperbolic curves.
\begin{thm}
\label{Theorem:xia_lift}Let $k$ be a perfect of characteristic $p$
and let $X$ be a smooth proper hyperbolic curve over $k$. Let $\mathscr{G}\rightarrow X$
be a height 2, dimension 1 BT group over $X$. If $\mathscr{G}$ is
everywhere versally deformed on $X$, then there is a unique curve
$\tilde{X}$ over $W(k)$ which is a lift of $X$ and admits a lift
$\tilde{\mathscr{G}}$ of $\mathscr{G}$. Furthermore, the lift $\tilde{\mathscr{G}}$
is unique.
\end{thm}

\begin{rem}
Note that the hypothesis of Theorem \ref{Theorem:xia_lift} implies
that $\mathscr{G}\rightarrow X$ is generically ordinary. This is
why we call it an analog to the Serre-Tate canonical lift.
\end{rem}

\begin{example}
\label{Example:shimura_igusa_not_everywhere_versally_deformed}Let
$D$ be a non-split quaternion algebra over $\mathbb{Q}$ that is
split at $\infty$ and let $p$ be a finite prime where $D$ splits.
The Shimura curve $X^{D}$ parametrizing fake elliptic curves for
$\mathcal{O}_{D}$ exists as a smooth complete curve (in the sense
of stacks) over $\mathbb{Z}[\frac{1}{2d}]$ and hence over $\mathbb{F}_{p}$.
(See Definition \ref{Definition:moduli_fake_elliptic_curves}.) Abusing
notation, we denote by $X^{D}$ the Shimura curve over $\mathbb{F}_{p}$.
It admits a universal abelian surface $f:\mathcal{A}\rightarrow X^{D}$
with multiplication by $\mathcal{O}_{D}$. As $D\otimes\mathbb{Q}_{p}\cong M_{2\times2}(\mathbb{Q}_{p}),$
one can use Morita equivalence, i.e. apply the idempotent
\[
\left(\begin{array}{cc}
1 & 0\\
0 & 0
\end{array}\right)
\]
on the height 4 dimension 1 BT group $\mathcal{A}[p^{\infty}]$ to
get a height 2 dimension 1 BT group $\mathscr{G}$, and in fact $\mathcal{A}[p^{\infty}]\cong\mathscr{G}\oplus\mathscr{G}$.
Here $\mathscr{G}$ is \emph{everywhere versally deformed} on $X^{D}$.
Moreover, $\mathscr{G}$ has both supersingular and ordinary points.
For similar discussion, see \cite[Section 5.1]{goren2005canonical}.

On the other hand, as in Example \ref{Example:modular_igusa_not_everywhere_versally_deformed},
let $X_{Ig}^{D}$ be the Igusa cover which trivializes $\mathscr{G}[p]^{et}$
on the ordinary locus. This cover is branched over the supersingular
locus of $\mathscr{G}$. Pulling back $\mathscr{G}$ to $X_{Ig}^{D}$
yields a BT group which is \emph{generically} \emph{versally deformed}
but not everywhere versally deformed.
\end{example}

\begin{cor}
\label{Corollary:p-div_shimura}Let $X\overset{f}{\leftarrow}Z\overset{g}{\rightarrow}X$
be an étale correspondence of projective hyperbolic curves without
a core over a perfect field $k$. Let $\mathscr{G}\rightarrow X$
be a BT group of height 2 and dimension 1 that is everywhere versally
deformed. Suppose further that $f^{*}\mathscr{G}\cong g^{*}\mathscr{G}$.
Then $X$ and $Z$ are the reduction modulo $p$ of Shimura curves.
\end{cor}

\begin{proof}
Any lift $\tilde{X}$ of $X$ naturally induces lifts $\tilde{Z}_{f}$
and $\tilde{Z}_{g}$ of $Z$ because $f$ and $g$ are étale and the
goal is to find a lift $\tilde{X}$ such that $\tilde{Z}_{f}\cong\tilde{Z}_{g}$.
Note that $f^{*}\mathscr{G}$ and $g^{*}\mathscr{G}$ are everywhere
versally deformed on $Z$ because $f$ and $g$ are étale. By Theorem
\ref{Theorem:xia_lift}, the pairs $(X,\mathscr{G})$, $(Z,f^{*}\mathscr{G})$,
and $(Z,g^{*}\mathscr{G}$) canonically lift to $W(k)$. As $f^{*}\mathscr{G}\cong g^{*}\mathscr{G}$,
the lifts of $(Z,f^{*}\mathscr{G})$ and $(Z,g^{*}\mathscr{G})$ are
isomorphic and we get an étale correspondence of curves over $W(k)$:
\[
\xymatrix{ & \tilde{Z}\ar[dl]_{\tilde{f}}\ar[dr]^{\tilde{g}}\\
\tilde{X} &  & \tilde{X}
}
\]
Finally by Lemma \ref{Lemma:lift_no_core}, $X$ and $Z$ are the
reductions modulo $p$ of Shimura curves.
\end{proof}
\begin{example}
Let $X=X^{D}$ be a Shimura curve parametrizing fake elliptic curves
with multiplication by $\mathcal{O}_{D}$, as in Definition \ref{Definition:moduli_fake_elliptic_curves},
with $D$ having discriminant $d$. Let $N$ be prime to $d$ and
let $Z=X_{0}^{D}(N)$, i.e. a moduli space of pairs $(A_{1}\rightarrow A_{2})$
of fake elliptic curves equipped with a ``cyclic isogeny of fake-degree
$N$'' \cite[Section 2.2]{bakker2013frey}. There exist Hecke correspondences
\[
\xymatrix{ & Z\ar[dl]_{\pi_{1}}\ar[dr]^{\pi_{2}}\\
X &  & X
}
\]

Moreover, as in Example \ref{Example:Fake_elliptic_curves_brauer},
as long as $(p,dN)=1$, this correspondence has good reduction at
$p$ and the universal $p$-divisible group splits as $\mathscr{G}\oplus\mathscr{G}$
on $X$. Here $\mathscr{G}\rightarrow X$ is everywhere versally deformed
and $\pi_{1}^{*}\mathscr{G}\cong\pi_{2}^{*}\mathscr{G}$. In particular,
there are examples where the conditions of Corollary \ref{Corollary:p-div_shimura}
are met.
\end{example}

\begin{rem}
\label{Remark:mochizuki} Xia's lifting theorem yelds one of the canonical
lifts of Mochizuki. In particular, as Shimura curves are defined over
$\overline{\mathbb{Q}}$, Corollary \ref{Corollary:p-div_shimura}
provides a condition under which the a canonical lift of $X$ is defined
over $\overline{\mathbb{Q}}$. This yields a very partial response
to \cite[Problem 10]{oortopen}. 

\end{rem}

\begin{thm}
\label{Theorem:2_pullbacks_L_isomorphic_versally}Let $X\overset{f}{\leftarrow}Z\overset{g}{\rightarrow}X$
be an étale correspondence of smooth, geometrically connected, complete
curves without a core over $\mathbb{F}_{q}$ with $q$ a square. Let
$\mathscr{L}$ be a $\overline{\mathbb{Q}}_{l}$-local system on $X$
as in Theorem \ref{Theorem:p-div_companion_versal-1} such that $f^{*}\mathscr{L}\cong g^{*}\mathscr{L}$
as local systems on $Z$. Suppose the $\mathscr{G}\rightarrow X$
constructed via Theorem \ref{Theorem:p-div_companion_versal-1} is
everywhere versally deformed. Then $X$ and $Z$ are the reductions
modulo $p$ of Shimura curves.
\end{thm}

\begin{proof}
The uniqueness statement in Theorem \ref{Theorem:p-div_companion_versal-1}
immediately implies that $f^{*}\mathscr{G}\cong g^{*}\mathscr{G}$.
Apply Corollary \ref{Corollary:p-div_shimura}.
\end{proof}
\begin{rem}
\label{Remark:strategy_fake_elliptic_curves}Finally, we explain a
strategy to show that in the context of Theorem \ref{Theorem:2_pullbacks_L_isomorphic_versally},
$X$ is the reduction modulo $p$ of a moduli space of fake elliptic
curves and furthermore that $\mathscr{L}$ is the universal local
system. In particular, this strategy will imply that under the hypotheses
of Theorem \ref{Theorem:2_pullbacks_L_isomorphic_versally}, $\mathscr{L}$
comes from a family of fake elliptic curves. This is joint work-in-progress
with M. Sheng.

The canonical lift constructed by Xia is \emph{uniformizing} in the
sense of Mochizuki \cite{mochizuki1996ordinary}. This notion has
recently been reinterpreted in the context of $p$-adic nonabelian
Hodge theory. In particular, in Theorem \ref{Theorem:xia_lift} consider
the generic fiber $\tilde{X}_{K(k)}$. By virtue of the existence
of the $p$-divisible group, we obtain a rank 2 crystalline representation
$\pi_{1}(\tilde{X}_{K(k)})\rightarrow\text{GL}_{2}(\mathbb{Z}_{p})$
which is uniformizing in the sense that under the $p$-adic Simpson
correspondence of \cite{LSZ13a}, the associated Higgs bundle is uniformizing. 

We briefly outline the strategy, which involves a general result about
Shimura curves. Let $\tilde{Y}$ be an integral model of a Shimura
curve over $W(k)$ associated to a totally real field $E$ and a quaternion
algebra $D$. (The totally real field $E$ is isomorphic to $\mathbb{Q}$
if and only if $\tilde{Y}$ is a moduli space of fake elliptic curves.)
We prove that there exists a positive integer $f$ such that the uniformizing
Higgs bundle on $\tilde{Y}$ is $f$-periodic using Deligne's theory
of strange models. In particular, the uniformizing Higgs bundle corresponds
to a crystalline representation $\pi_{1}(\tilde{Y}_{K(k)})\rightarrow\text{GL}_{2}(\mathbb{Z}_{p^{f}})$,
or equivalently a Fontaine-Faltings module with endomorphism structure
$(V,\nabla,\varphi,\text{Fil},\iota)$. By forgetting the filtration,
one obtains an $F$ crystal in finite, locally free modules and multiplication
by $\mathbb{Z}_{p^{f}}$ on the special fiber $Y$. We must simply
show that the field generated by the Frobenius traces of the associated
object of $\textbf{F-Isoc}(Y)_{\mathbb{Q}_{p^{f}}}$ is isomorphic
to $E\subset\mathbb{Q}_{p^{f}}$, the reflex field of the Shimura
curve. This should again follow from the explicit computation with
strange models.

Given the above statement about the Frobenius traces, it will follow
that in the context of Theorem \ref{Theorem:2_pullbacks_L_isomorphic_versally},
$(X,\mathscr{L})$ comes from a universal family of fake elliptic
curves together with the associated local system. We emphasize that
while this strategy will only work in the restrictive context of Theorem
\ref{Theorem:2_pullbacks_L_isomorphic_versally}, it would provide
some evidence for Conjecture \ref{Conjecture:rank_2_fake_elliptic_curve}.
\end{rem}

\bibliographystyle{alphaurl}
\bibliography{descent_new}

Bergische Universität Wuppertal

F13.05

Gaußstraße 20, Wuppertal\\

\email{krishnamoorthy@alum.mit.edu}
\end{document}